\documentclass[a4paper,12pt]{article}
\usepackage{hyperref}
\usepackage[latin1]{inputenc}
\usepackage{amsmath,amsthm,amssymb,amsfonts}
\usepackage{enumerate,enumitem}
\usepackage{graphics,graphicx,subfigure}
\usepackage{mathrsfs} 

\evensidemargin = \oddsidemargin
\advance\textwidth by -2in
\theoremstyle{plain}
\newtheorem{theoreme}{Théorème}[section]
\newtheorem{theorem}[theoreme]{Theorem}

\newtheorem*{theoremeA}{Theorem A}
\newtheorem*{theoremeB}{Theorem B}
\newtheorem*{theoremeB'}{Théorème B'}

\newtheorem{proposition}[theoreme]{Proposition}
\newtheorem{question}{Question}

\newtheorem{lemme}[theoreme]{Lemme}
\newtheorem{lemma}[theoreme]{Lemma}

\newtheorem{corollary}[theoreme]{Corollary}

\theoremstyle{definition}
\newtheorem{definition}[theoreme]{Definition}

\newtheorem{claim}[theoreme]{Claim}
\theoremstyle{remark}
\newtheorem{remarque}[theoreme]{Remarque}
\newtheorem{remark}[theoreme]{Remark}

\def\N{{\mathbb N}}
\def\Q{{\mathbb Q}}
\def\R{{\mathbb R}}
\def\Z{{\mathbb Z}}
\def\T{{\mathbb T}}
\def\eps{\varepsilon}
\def\f{\varphi}

\def\CC{{\mathcal C}}

\def\Diff{{\mathrm {Diff}}}
\def\DD{{\mathrm {Diff}}^\infty_+[0,1]}

\def\DN{{\mathrm {Diff}}^\infty_+N}
\def\DS{{\mathrm {Diff}}^\infty_+\T^1}
\def\DT{{\mathrm {D}}^\infty_+\T^1}
\def\DTo{{\mathrm {D}}^0_+\T^1}
\def\Ds{{\mathrm {Diff}}^\star[0,1]}
\def\Dsp{{\mathrm {Diff}}^\star_\partial[0,1]}

\def\DJ{{\mathrm {Diff}}^\infty_+J}
\def\DJs{{\mathrm {Diff}}^\star J}
\def\DJsp{{\mathrm {Diff}}^\star_\partial J}

\def\Cinf{{\CC^{\infty}}}

\def\RR{{\mathcal R}}
\def\RC{{\mathcal R^{pc}}}

\def\RRd{{\mathcal R^{[0,1]}_2}}
\def\RRn{{\mathcal R^{[0,1]}_n}}
\def\RRns{{\mathcal R^{\star}_{n,\partial}}}
\def\RNn{{\mathcal R^{N}_n}}

\def\RSd{{\mathcal R^{\T^1}_2}}
\def\RSdnl{{\mathcal R^{\T^1}_{2,\mathrm{nf}}}}

\def\RSn{{\mathcal R^{\T^1}_n}}
\def\RSnnl{{\mathcal R^{\T^1}_{n,\mathrm{nf}}}}

\def\RJ{{\mathcal R^{J}}}

\def\Rv{{\mathcal R^{\star}_{\varnothing}}}
\def\RCv{{\mathcal R^c_{\varnothing}}}

\def\RJs{{\mathcal R^{J,\star}}}
\def\RJsp{{\mathcal R^{J,\star}_{\partial}}}
\def\RIbs{{\mathcal R^{\bar{I},\star}_{\partial}}}
\def\Rss{{\mathcal R^{\star}}}
\def\Rs{{\mathcal R^{\star}_\partial}}
\def\RCs{{\mathcal R^c}}
\def\RCp{{\mathcal R^c_{\partial}}}

\def\RCJsp{{\mathcal R^{J,c}_{\partial}}}
\def\RCJs{{\mathcal R^{J,c}}}
\def\RCIbs{{\mathcal R^{\bar{I},c}_{\partial}}}

\def\Rsi{{\mathcal R^{\star}_{\partial,\R\setminus\Q}}}
\def\RCsi{{\mathcal R^c_{\partial,\R\setminus\Q}}}

\def\Rsr{{\mathcal R^{\star}_{\partial,\Q}}}

\def\Rid{{\mathcal C_{\id}}}

\def\ft{{\tilde{f}}}

\def\gt{{\tilde{g}}}
\def\xit{{\tilde{\xi}}}

\def\xib{{\bar{\xi}}}
\def\fb{{\bar{f}}}
\def\gb{{\bar{g}}}

\newcommand{\id}{\mathop{\mathrm{id}}\nolimits}
 
\def\Fix{\mathop{\mathrm{Fix}}} 
 
\def\ITI{\mathop{\mathrm{ITI}}}

\def\ly{\fontencoding{U}\fontfamily{lasy}\fontseries{m}\fontshape{n}\selectfont}
\def\guil#1{\leavevmode\hbox{{\ly(\kern-0.20em(\kern+0.20em}}\nobreak{}\,#1\,%
  \nobreak\hbox{{\ly\kern+0.20em)\kern-0.20em)}}}

\def\norm#1{\left\lVert #1 \right\rVert}
\def\res#1{\mathbin{|}{}_{#1}}

\title{Connectedness of the space of smooth actions of $\Z^n$
on the interval}

\begin{document}

\author{C.~Bonatti, H.~Eynard}

\date{}

\maketitle

\begin{abstract}
We prove that the spaces of $\Cinf$ orientation-preserving actions of $\Z^n$ on $[0,1]$ and nonfree actions of $\Z^2$ on the circle are connected. 
\end{abstract}

\vskip 5mm
\centerline{\LARGE{Connexité de l'espace des actions $\CC^{\infty}$ de $\Z^n$ sur l'intervalle}}

\vskip 5mm
\centerline{\footnotesize{{\bf R\'esum\'e.}}}

\noindent
\footnotesize{ Nous montrons que l'espace des actions $\CC^{\infty}$, pr\'eservant l'orientation, de $\Z^n$ sur l'intervalle $[0,1]$ est connexe, de m\^eme que l'espace des actions non-libres de $\Z^2$ sur le cercle.}

\renewcommand{\thefootnote}{}

\footnote{This paper has been partially supported by the ANR project  BLAN08-2 313375, DYNNONHYP}


\section{Introduction}
\subsection{General setting}
The theory of foliations is closely related to the study of smooth actions of a finitely generated group. For instance, given closed manifolds $M,N$, any action $\rho\colon \pi_1(M)\to \Diff(N)$ of the fundamental group $\pi_1(M)$ on $N$ induces, by \emph{suspension}, a foliation transverse to a $N$-fiber bundle over $M$ whose holonomy group is $\rho(\pi_1(M))$.  More generally, the transverse structure of a foliation is decribed by the action of a pseudo-group of diffeomorphisms.

Foliations and group actions are natural generalisations of dynamical systems, and their studies share many concepts, questions and arguments.  However, the study of dynamical systems uses in a fundamental way perturbation arguments.  The idea is that describing the behavior of every system may be hopeless, but small perturbations could avoid fragile pathological behaviors. 
Very few works in foliation theory use this strategy (see however \cite{BoFi, Br, Ts}). The main reason is that there are very few perturbation lemmas and that we don't know in general how to perform a non trivial perturbation of a foliation or of a group action.To be concrete, here, we focus on actions of $\Z^n$ on the segment $[0,1]$ or on the circle $\T^1=\R/\Z$. In other words we consider $n$-tuples of commuting diffeomorphisms $(f_1,\dots, f_n)$. One can easily perturb each of the $f_i$ but keeping the commutative relations leads to important difficulty. Let us illustrate this difficulty by an example. One can find in \cite[1975]{RoRo} the following sentence:

\emph{Thus G is a foliation of $\T^3$ by planes $\R^2$, and we must
admit (with much consternation) that we do not know if a stable foliation of
$\T^3$ by planes does exist. This is an important problem.}
%

The underlying problem can be rephrased as follows:
\begin{question} Does there exist $r\geq 1$ and a pair of commuting $C^r$-diffeomorphisms $f,g$ of $\T^1$ with irrational rotation numbers such that any commuting diffeomorphisms $\tilde f,\tilde g$ sufficiently $C^r$-close to $(f,g)$ are jointly (topologically) conjugated to $(f,g)$?  
\end{question} 
As far as we know, this question remains open for any $C^r$-topology. 
In the same spirit, the following question (also attributed to Rosenberg) can be found in \cite{Yo}:
\begin{question}\label{q.cercle} Consider the subset of $\Diff^r(\T^1)^2$ consisting of commuting diffeomorphisms. Is it locally connected?  
\end{question}
This question too remains open for any topology. It was one of the motivations for the deep understanding of the diffeomorphisms of the circle, by M. Herman and J-C. Yoccoz. However, their works did not provide an answer to this precise question which remained ununderstood for many years. Note, however, that the (global) connectedness of the space of $C^1$-actions of $\Z^n$ on $\T^1$ has been proved by A. Navas very recently \cite{Na2}. 

This question is intimately related to the question of local connectedness of the space of $C^r$ codimension $1$ foliations on a given manifold. In 2007, A. Larcanché \cite{La} presented a first relevant progress in that respect, concerning $3$-manifolds. The precise question was: 

\begin{question}\label{q.connexite}  Is there  a
one-to-one correspondance between the connected components of the space of $C^r$-codimension one foliations, $r\in\{1,\dots,\infty\}$, on a given closed $3$-manifold and those of the space of plane fields ?
\end{question}
The fact that any homotopy class of plane fields contains a foliation is due J. Wood \cite{Wo}. 
Larcanché showed that codimension $1$ foliations of a $3$-manifold \emph{in the neighborhood of a taut foliation} can be connected by  (long) smooth pathes of foliations.  
Then, in her thesis \cite{Ey1}, the second author reduced Question~\ref{q.connexite} for $\Cinf$-foliations to the problem of the connectedness of the space of $\Z^2$ $\Cinf$-actions on $[0,1]$. The aim of this paper is to solve this last question, providing the last piece that was missing in \cite{Ey1}. Therefore, the (positive) solution of Question~\ref{q.connexite} for $\Cinf$-foliation will be the aim of \cite{Ey4}. 

\subsection{Precise results}
Let $N$ denote either the segment $[0,1]$ or the circle $\T^1=\R/\Z$, and $\DN$ the group of $\Cinf$ orientation preserving diffeomorphisms of $N$.
A representation of $\Z^n$ into $\DN$ is nothing but the data of $n$ commuting $\Cinf$ diffeomorphisms of $N$. Thus, the space of such representations will be regarded here as the subspace
$$\RNn=\{(f_1,...,f_n)\in(\DN)^n : f_i\circ f_j = f_j\circ f_i\;\forall \;1\le i,j\le n\}\subset (\DN)^n$$
equipped with the induced topology, where $\DN$ is endowed with the usual $\Cinf$ topology.

\begin{theoremeA}\label{t:principal}
The space $\RRn$ of representations of $\Z^n$ into $\DD$ is connected. More precisely, the path-connected component $\Rid$ of $(\id,...,\id)$ is dense in $\RRn$.
\end{theoremeA}

\begin{remarque}\begin{enumerate}[label=(\emph{\roman*})]
\item This theorem says nothing about the \emph{path-connectedness} or \emph{local connectedness} of this space, which remain open questions. 
\item \cite{Ey2} contained a weaker result: any representation of $\Z^n$ into $\DD$ can be connected to $(\id,...,\id)$ by a path of $\CC^1$ representations (continuous for the $\CC^1$ topology). This has recently been greatly improved by A. Navas \cite{Na2} who proved that the space of  representations of $\Z^n$ into $\Diff^1_+[0,1]$ (endowed with the $\CC^1$ topology) is connected. While the techniques used in the present article are very similar to those used in \cite{Ey2} (based on the rigidity of the commutativity condition in regularity at least $\CC^2$), they are very different from those used in the $\CC^1$ regularity case in \cite{Na2}.
\item Our proof really uses \emph{infinite} differentiability and does not seem to adapt to finite differentiability (even $\ge 2$). 
\item It is not hard to see, using some ``Alexander's trick"-like argument, that the space of orientation-preserving $\CC^0$ actions of $\Z^n$ on $[0,1]$ is contractible. 
\end{enumerate}
\end{remarque}

\begin{proof}[Idea of the proof.] The cases $n=0$ and $n=1$ are trivial since the space $\DD$ is contractible. Let us now consider the case $n=2$. Of course, deforming a given pair of diffeomorphisms $(f,g)$ into another one is not difficult if one forgets about the commutativity condition, but this constraint adds a lot of rigidity to the problem. Namely, if we restrict to the case of diffeomorphisms $f$, $g$ which are nowhere infinitely tangent to the identity in $(0,1)$ (such pairs will be referred to as ``nondegenerate" in Section 2), classical results by N. Kopell \cite{Ko}, G. Szekeres \cite{Sz} and F. Takens \cite{Ta} (cf. \ref{ss:classic})  imply that $f$ and $g$ belong either to a common infinite cyclic group generated by some $\Cinf$ diffeomorphism $h$ of $[0,1]$ or to a common $\CC^1$ flow ($\CC^\infty$ on $(0,1)$ but in general not $\CC^2$ on $[0,1]$). Then, our strategy is as follows. 

\begin{itemize}
\item In the first case, any isotopy $t\in[0,1]\mapsto h_t$ from $\id$ to $h$ yields a path $t\mapsto (h_t^p,h_t^q)$ of commuting
 $\Cinf$-diffeomorphisms from $(\id,\id)$ to $(f=h^p,g=h^q)$, so $(f,g)$ is actually \emph{in} the path-connected component of $(id,id)$ ($\Rid$) and we have nothing to do. 
\item In the second case, however, extra-work is called for. If $f$ and $g$ are the time-$\alpha$ and $\beta$ maps of a $\CC^1$ vector field $\xi$ ($\CC^\infty$ on $(0,1)$), the idea is to construct a $\Cinf$ vector field $\tilde{\xi}$ whose time-$\alpha$ and $\beta$ maps $\f^\alpha$ and $\f^\beta$ are arbitrarily $\Cinf$ close to $f$ and $g$ respectively. The pair $(\f^\alpha,\f^\beta)$ is then easily connected to $(\id,\id)$ by a continuous path of pairs of commuting $\Cinf$ diffeomorphisms $t\in[0,1]\mapsto (\f^{t\alpha},\f^{t\beta})$. One can then conclude that $(f,g)$ belongs to the closure of $\Rid$. 
\end{itemize}
In other words, what we show is that, among ``nondegenerate" pairs, those made of iterates of a same smooth diffeomorphism or of elements of a same smooth flow (we call such pairs ``clean" in section 2) form a dense and path-connected subset.\medskip

Then, deriving the general result from the restricted one we just mentionned is elementary and is done in section \ref{ss:general}. The case $n>2$ is similar (cf. \ref{ss:Zn}).\medskip

The strategy seems very simple. \emph{But} let us stress that, in the second case above, a random smoothing of the vector field $\xi$ near the boundary won't do in general, for the resulting flow would be no more than $\CC^1$ close to that of $\xi$. So first, one needs to derive some nice estimates on $\xi$ from the knowledge that \emph{some} times of its flow are $\Cinf$ (cf. \ref{ss:approx} and \ref{s:est}). More precisely, if $\xi$ is not $\CC^\infty$ near a point of the boundary, say $0$, according to Takens \cite{Ta}, $f$ and $g$ are necessarily infinitely tangent to the identity at that point. What we show in that case is that, though the derivatives of $\xi$ globally diverge when one approaches $0$, arbitrarily close to $0$, one can find whole intervals (disjoint from $0$) where they are arbitrarily small. These estimates are a generalization of those obtained by F. Sergeraert in \cite{Se} and constitute the heart of this article. Then the rough idea to construct $\xit$ is simply to replace $\xi$ between $0$ and such a ``nice interval" by something smooth \emph{and} ``$\CC^\infty$-small" (the latter being made possible precisely by the estimates on $\xi$ in the ``nice interval"), leaving it unchanged outside this small region. Then the time-$\alpha$ and $\beta$ maps of the new vector field $\xit$ basically coincide with $f$ and $g$ away from the boundary and are very close to the identity there, as are $f$ and $g$ ! \medskip

Note, to conclude, that what our strategy provides in the situation above is an \emph{approximation} of $(f,g)$ by clean pairs, not a \emph{continuous deformation}, simply because between the ``nice intervals" essential to our construction lie  ``nasty" ones. Precisely there lies the gap between connectedness and \emph{path}-connectedness.

\end{proof}

We can then derive a similar result on the circle, but only for the subspace of $\RSn$ made of the $\Z^n$-actions which admit no free sub $\Z^2$-action:
$$\RSdnl=\{(f,g)\in(\DS)^2 : f\circ g = g\circ f\,\text{and}\, \exists (p,q,x)\in\Z^2\setminus\{(0,0)\}\times\T^1\,\text{s.t.}\,f^q\circ g^p(x)=x\}$$
and
$$\RSnnl=\{(f_1,...,f_n)\in(\DS)^n : (f_i,f_j)\in\RSdnl \;\forall \;1\le i,j\le n\}.$$

\begin{theoremeB}\label{t:b}
The space $\RSnnl$ is connected. More precisely, the path-connected component $\Rid$ of $(\id,...,\id)$ is dense in $\RSnnl$.

In other words, every smooth $\Z^n$-action on $\T^1$ without free sub $\Z^2$-action   can be $\CC^\infty$-approached by smooth $\Z^n$-actions isotopic to the identity. 
\end{theoremeB}

\begin{remarque}
The nonfreeness condition on the action of $\Z^2$ on $\T^1$ defined by a pair $(f,g)$ is equivalent to the existence of $(p,q)\in\Z^2\setminus\{(0,0)\}$ such that $p\rho(f)+q\rho(g)\in\Z$, where $\rho(f)$ and $\rho(g)$ denote the rotation numbers of $f$ and $g$ (cf. Notations and formulae).
\end{remarque}

\subsection{Conclusion}
Let us conclude this introduction by summerizing what remains to be done in order to get the connectedness of the  space of smooth $\Z^2$-actions on $\T^1$.

If, contrary to our setting, the action defined by $(f,g)\in\RSd$ is free, $\rho(f)$ and $\rho(g)$ are irrational numbers and either do or do not satisfy a joint diophantine condition (cf.~\cite{F-K} for the precise meaning of this).
\begin{itemize}
\item If they do, B.~Fayad and K.~Khanin \cite{F-K} proved that $f$ and $g$ are simultaneously conjugate to the rotations of angle $\rho(f)$ and $\rho(g)$, denoted by $R_{\rho(f)}$ and $R_{\rho(g)}$ respectively, by an element $\f$ of $\DS$. The pair $(f,g)$ is thus connected to $(\id,\id)$ by the path $t\in[0,1]\mapsto (\f^{-1}\circ R_{t\rho(f)}\circ \f, \f^{-1}\circ R_{t\rho(g)}\circ \f)$ of smooth commuting diffeomorphisms. 
\item If they do not, $(f,g)$ is not necessarily smoothly conjugate to $(R_{\rho(f)},R_{\rho(g)})$. Nevertheless, according to M. Benhenda \cite{Be}, there exists a Baire-dense subset $B$ of $\T^1$ such that, if $\rho(f)$ \emph{or} $\rho(g)$ belongs to $B$, $(f,g)$ can be \emph{approached} by pairs which \emph{are} smoothly conjugate to $(R_{\rho(f)},R_{\rho(g)})$. Thus $(f,g)$ belongs to the \emph{closure} of the path-connected component of $(\id,\id)$. 
\end{itemize}

It is not known, however, whether this last fact holds for \emph{any} pair $(\rho(f),\rho(g))\in(\R\setminus\Q)^2$. A positive answer would imply the connectedness of the whole space of $\Z^2$-actions on the circle.

\begin{remark}
In this general case, what is already known from \cite{Yo} is that $f$ and $g$ can be approached \emph{seperately} by diffeomorphisms which are smoothly conjugate to $R_{\rho(f)}$ and $R_{\rho(g)}$ respectively. What is not known is whether the same conjugating diffeomorphism can be used for $f$ and $g$.
\end{remark}

\section*{Notations and formulae}

\noindent\textbf{Derivatives.} Let $k\in \N$, $I$ be an interval of $\R$, $f$ a $\CC^k$ function on $I$ and $J$ a subset of $I$. We write
$$\norm{f}_{k} = \underset{{x\in I}\atop{0\le i\le k}}{\sup} |D^if(x)|\in[0,+\infty]$$
and
$$\norm{f}_{k,J} = \underset{{x\in J}\atop{0\le i\le k}}{\sup} |D^if(x)|\in[0,+\infty].$$
If $g$ is a $\CC^2$ orientation preserving diffeomorphism of $I$, we define :
\begin{equation*}
Lg = D \log Dg = \frac{ D^2 g }{ Dg } .
\end{equation*}
The operator $L$ satisfies the following ``derivation" formulas:
\begin{equation} \label{e:chainrule}
   L (h \circ g) = Lh \circ g \cdot Dg + Lg \quad\text{and}\quad    L g^k = \sum_{i=0}^{k-1}Lg \circ g^i \cdot Dg^i.
\end{equation}

A function is said to be $\CC^k$-\emph{flat} (resp. \emph{infinitely flat}) at some point if its derivatives of order $0$ to $k$ (resp. all its derivatives) exist and vanish at this point. \medskip

\noindent\textbf{Fixed points.} We denote by $\Fix(f)$ the set of fixed points of a diffeomorphism $f$. We say that $f$ is $\CC^k$-\emph{tangent to the identity} (resp. \emph{infinitely tangent to the identity}) at a point if $f-\id$ is $\CC^k$-flat (resp. infinitely flat) at this point. We will abbreviate ``infinitely tangent to the identity" by ``ITI", and denote by $\ITI(f)$ the set of points where $f$ is ITI.\medskip


\noindent\textbf{Diffeomorphisms of the circle.} We denote by $\DT$ (resp. $\DTo$) the space of $\Cinf$ orientation preserving diffeomorphisms (resp. homeo\-mor\-phisms) of $\R$ which commute to the unit translation $x\mapsto x+1$, both endowed with their usual topology. An element $f$ of $\DT$ naturally projects to an element $\pi(f)$ of $\DS$. One can consider $\DT$ together with this projection $\pi$ as the universal cover of $\DS$. Note that if the projections $\pi(f)$ and $\pi(g)$ of two elements of $\DT$ commute, then so do $f$ and $g$. 


We denote by $R_\alpha\in\DT$ the translation $x\mapsto x+\alpha$. Recall that when $\alpha\in\R\setminus \Q$, the centralizer of $R_\alpha$ in $\DTo$ is the group of translations $(R_\beta)_{\beta\in\R}$. 

For all $f\in\DT$ (resp. $\in\DS$), let $\rho(f)\in\R$ (resp. $\T^1$) denote the rotation number of $f$. 
The rotation number of $f$ equals $p/q$, $(p,q)\in\Z\times \N^*$, $p\wedge q=1$, if and only if the $1$-periodic map $f^q-\id-p$ vanishes somewhere, and $f$ is $\Cinf$ conjugate to the rotation $R_{p/q}$ if and only if this function is everywhere zero. If $\rho(f)\in\R\setminus\Q$, Denjoy proved that $f$ is $\CC^0$-conjugate to $R_{\rho(f)}$. \medskip

\noindent\textbf{Vector fields vs. functions.} We will often make no difference between a vector field $\nu\partial x$ on $I\subset \R$ and the  function $\nu$ on the same interval, $x$ denoting the coordinate on $\R$, and, consequently, between the pull-back $h^*(\nu \partial x)$ of $\nu\partial x$ by a diffeomorphism $h$ and the function $\nu\circ h/Dh$.


\section{$\Z^n$-actions on the segment}

Most of this section is devoted to the proof of the connectedness of the space $\RRd$ of smooth orientation-preserving $\Z^2$-actions on $[0,1]$, i.e Theorem A for $n=2$. We explain in \ref{ss:Zn} how to extend it to all $n\ge 2$. For now, let us abbreviate $\RRd$ by $\RR$. The first (and key) step towards Theorem A is to prove the connectedness of a particular subclass $\Rss\subset \RR$ of ``non-degenerate" representations (cf. \ref{ss:classic} and \ref{ss:approx}). The generalization to the whole space $\RR$ is then carried out in \ref{ss:general}.

\subsection{Rigidity results for commuting interval diffeomorphisms}\label{ss:classic}
 Before introducing $\Rss$, let us recall some classical facts about commuting interval diffeomorphisms. Along with the original references, we refer the reader to \cite{Na1} and \cite{Yo} for detailed proofs and much more. 

\subsubsection{Original statements}\label{sss:kst}
Let $a<b$ be two real numbers.
\begin{theorem}[Szekeres \cite{Sz}]\label{t:Szekeres}
Every diffeomorphism $f\in \Diff^r[a,b)$, $r \ge 2$, without fixed point in
$(a,b)$ is the time-$1$ map of a vector field on $[a,b)$ which is $\CC^1$ on
$[a,b)$ and $\CC^{r-1}$ on $(a,b)$.
\end{theorem}

\begin{remark} In general, ``$\CC^1$" cannot be replaced by ``$\CC^{r-1}$" in the above statement. Indeed, in \cite{Se}, F. Sergeraert constructed an $f$ in $\Diff^\infty[a,b)$ without fixed point in
$(a,b)$ and which \emph{does not} imbed in a $\CC^2$ flow. 
\end{remark}

\begin{theorem}[Kopell's Lemma \cite{Ko}]\label{t:Kopell}
Let $f$ and $g$ be commuting diffeomorphisms of $[a,b)$ of class $\CC^2$ and $\CC^1$ respectively. If $f$ has no fixed point in $(a,b)$ and $g$ has at least one, then $g = \id$. 
\end{theorem}

\begin{corollary}\label{t:Z1}
Let $f \in \Diff^r[a,b)$, $r \ge 2$, without fixed point in $(a,b)$. There exists a unique $\CC^1$-vector field on $[a,b)$ having $f$ as time-$1$ map. We call it \emph{the Szekeres vector field of $f$} and denote it by $\xi_f^{[a,b)}$. The centralizer of $f$ in $\Diff^1[a,b)$ consists of the flow maps of this vector field and is thus a one-parameter group of $\CC^1$-diffeomorphisms.
\end{corollary}

Of course, similar results hold for diffeomorphisms of $(a,b]$.\medskip

Thus, if $g\in \Diff^1[a,b)$ commutes with such an $f\in \Diff^r[a,b)$, it is the time-$\tau$ map of the Szekeres vector field of $f$ for some $\tau\in\R$, which we refer to as the \emph{translation number of $g$ with respect to $f$}. In the next section, we will see that this notion of translation number extends to certain pairs $(f,g)\in\RR$ under some non-degeneracy condition. But beforehand, let us introduce one last classical result:

\begin{theorem}[Takens] \label{t:Takens}
If $f \in \Diff^\infty_+ [b,c)$ has no fixed point in $(b,c)$ and is not ITI at $b$, its Szekeres vector field $\xi_f^{[b,c)}$ is smooth on all of $[b,c)$. 

Furthermore, if $f \in \Diff^\infty_+ (a,c)$ has a unique fixed point $b$ in $(a,c)$ and is not ITI there, 
the \emph{smooth} vector fields $\xi_f^{(a,b]}$ and $\xi_f^{[b,c)}$ match up smoothly at $b$ and yield a $\Cinf$ vector field on
$(a,c)$ whose time-$1$ map is $f$. Moreover, for any $\CC^\infty$-diffeomorphism of
$(a,c)$ commuting with $f$ (necessarily fixing $c$), and thus, according to \ref{t:Z1}, coinciding on $(a,b]$  (resp. $[b,c)$)  with some time-$\tau$ (resp. $\tau'$) map of $\xi_f^{(a,b]}$ (resp. $\xi_f^{[b,c)}$), the times $\tau$ and $\tau'$ coincide. 
\end{theorem}

\subsubsection{Nondegenerate representations}\label{sss:ndr}
Note that if $f\in\DD$, for any connected component $(a,b)$ of $[0,1]\setminus \Fix(f)$, all the statements of \ref{sss:kst} above apply both to $f\res{[a,b)}$ and $f\res{(a,b]}$, and one can show the following:

\begin{corollary}\label{c:kst}
Let $(f,g)\in\RR$, and $J$ be the closure of a connected component of $[0,1]\setminus \ITI(f)\cap\ITI(g)$. If $f\res{J}$ and $g\res{J}$ differ from the identity, they have the same fixed points in the interior of $J$, and the same finite order of contact to the identity at each of them.
\end{corollary}
To prove the connectedness of $\RR$, the basic idea is to show that (for a given $J$) commuting pairs such as $(f\res{J},g\res{J})$ above (i.e. without common ITI fixed point in the interior of $J$) form, together with the trivial pair $(\id_J,\id_J)$, a connected space. This is the subject of \ref{ss:classic}, \ref{ss:approx} and \ref{s:est}. Of course, it is sufficient to deal with the case $J=[0,1]$. So let us denote by $\Rss$ the  subspace of $\RR$ made of $(\id,\id)$ and of all pairs without common ITI fixed point in $(0,1)$. We call such pairs ``nondegenerate". According to \ref{c:kst} above, 
\begin{equation}\label{e:Rs}
\Rss=\RR\cap (\Ds)^2,
\end{equation}
where $\Ds\subset\DD$ consists of the diffeomorphisms which are nowhere ITI in $(0,1)$ unless they are the identity. 

As a matter of fact, we will need a little more than the connectedness of $\Rss$ to obtain that of $\RR$: in order to ``patch things up together nicely" at the end, we need to pay attention to what happens \emph{at the boundary} of the segment along our process. Therefore we introduce yet another notation: for any subset $\partial\subset\{0,1\}$, we denote by $\Dsp\subset\Ds$ the set made of the diffeomorphisms which are ITI exactly at the points of $\partial$, together with the identity, and define
$$\Rs=\Rss\cap (\Dsp)^2.$$
%

\subsubsection{Global relative translation number and Szekeres vector field for nondegenerate representations}\label{ss:sg}

Proposition  \ref{t:kst} below gives a concrete description of $\Rss$, summerized in \ref{c2:kst}. In particular, it shows that the notion of relative translation number introduced in \ref{sss:kst} extends naturally to nondegenerate pairs of commuting diffeomorphisms of $[0,1]$. 
%

This proposition mainly follows from the above-mentionned results by Kopell \cite{Ko}, Szekeres \cite{Sz} and Takens \cite{Ta}. Only point $(iii)$ is not a consequence of \ref{t:Z1} and \ref{t:Takens} above and requires some more work based on an observation by Yoccoz \cite{Yo}. 
We will only give an idea of the proof; the details can be found in \cite{Ey1} for example.

\begin{proposition}\label{t:kst} Let $(f,g)\in\Rs$ for some $\partial\subset\{0,1\}$. If $f$ and $g$ differ from the identity, 
\begin{enumerate}[label=(\roman*)]
\item $f$ and $g$ have exactly the same set $F$ of fixed points in $(0,1)$, with the same (finite) order of contact to the identity;\smallskip

\item there exists a unique $\alpha\in \R^*$ such that, on every connected component $(a,b)$ of $[0,1]\setminus F$, $g$ coincides with the time-$\alpha$ maps of $\xi_f^{[a,b)}$ and $\xi_f^{(a,b]}$. We say that $\alpha$ is the \emph{relative translation number} of $g$ with respect to $f$, and denote it by $\tau_{g/f}$; \smallskip

\item if $\alpha$ is irrational, on every connected component $(a,b)$ of $[0,1]\setminus F$, $\xi_f^{[a,b)}$ and $\xi_f^{(a,b]}$ coincide. We can then define the \emph{global} Szekeres vector field of $f$ on $[0,1]$ as the vector field $\xi_f$ vanishing on $F$ and coinciding with $\xi_f^{[a,b)}$ and $\xi_f^{(a,b]}$ on each connected component $(a,b)$ of $[0,1]\setminus F$. This vector field is $\CC^1$ on $[0,1]$, $\Cinf$ on $[0,1]\setminus \partial$ and nowhere infinitely flat there. \smallskip

\item if $\alpha = p/q$ with $(p,q)\in\Z^*\times \N^*$, $p\wedge q=1$, $f$ and $g$ are the $q$-th and $p$-th iterates of a common $\Cinf$-diffeomorphism $h$ of $[0,1]$, ITI exactly on $\partial$, which coincides on every connected component $(a,b)$ of $[0,1]\setminus F$ with the time-$1/q$ maps of $\xi_f^{[a,b)}$ and $\xi_f^{(a,b]}$ (in this case, these vector fields do not necessarily coincide).
\end{enumerate}\smallskip

If $f\neq \id$ and $g=\id$, we write $\tau_{g/f}=0$. Conversely, if $f=\id$ and $g\neq \id$, we write $\tau_{g/f}=\infty$.
\end{proposition}

Let us summerize this as follows:

\begin{corollary}\label{c2:kst} Let $(f,g)\in\Rss$. Then $f$ and $g$ belong either to a common infinite cyclic group generated by a $\Cinf$ diffeomorphism of $[0,1]$, or to a common $\CC^1$ flow ($\Cinf$ on $(0,1)$).
\end{corollary}


\begin{remark} \label{r:equ} \begin{enumerate}[label=(\emph{\roman*})]
\item Of course, given $(f,g)\in\Rss\setminus\{(\id,\id)\}$, one can define $\tau_{f/g}$ in the same way and check that $\tau_{f/g}=(\tau_{g/f})^{-1}$. Indeed, in the situation of $(i)$ in \ref{t:kst}, by uniqueness of the (standard) Szekeres vector field on a semi-open interval (cf.~\ref{t:Z1}), $\xi_g^{[a,b)}$ and $\xi_g^{(a,b]}$ are none but $\tau_{g/f}\xi_f^{[a,b)}$ and $\tau_{g/f}\xi_f^{(a,b]}$ , and $f\res{(a,b)}$ coincides with the time-$(\tau_{g/f})^{-1}$ maps of both. 
\item In \ref{t:kst}$(iii)$,  even though the time-$t$ maps of $\xi_f$ are $\Cinf$ on all of $[0,1]$ for all $t$ in some dense subset $\Z+\alpha \Z$ of $\R$, the vector field $\xi_f$ is no more than $\CC^1$ at $\partial$ in general (cf. \cite{Ey3}).
\end{enumerate}
\end{remark}

\begin{proof}[Idea of the proof of \ref{t:kst}]
$(i)$ comes from corollary \ref{c:kst}. \medskip

In $(ii)$, the fact that $g$ belongs to the flow of both Szekeres vector fields of $f$ between two consecutive fixed points comes from Corollary \ref{t:Z1}. That the corresponding time is the same for the left and right vector fields is a matter of relative combinatorics of the orbits of $f$ and $g$. Finally, that the time does not depend on the connected component of $[0,1]\setminus \Fix(f)$ comes from Takens's Theorem \ref{t:Takens}, since $f$ is nowhere ITI in $(0,1)$.\medskip

In $(iii)$, the equality between the left and right Szekeres vector fields is elementary: their time-$1$ and $\alpha$ maps coincide, so the same holds for their time-$(p+\alpha q)$ maps for all $p,q\in\Z$, and extends to all $t\in\R$ by continuity of the flow and density of $\Z+\alpha\Z$ in $\R$, $\alpha$ being irrational. The $\Cinf$ regularity of $\xi_f$ on $(0,1)$ then comes from Takens \ref{t:Takens}. The $\CC^1$ regularity at $0$ and $1$ is a consequence of an observation by Yoccoz \cite{Yo} concerning the continuous dependance of the Szekeres vector field (in $\CC^1$ topology) with respect to its time-$1$ map (in $\CC^2$ topology): if $0$ (resp. $1$) is an isolated fixed point, $\CC^1$ regularity of $\xi_f$ at $0$ simply comes from Szekeres' Theorem \ref{t:Szekeres}; otherwise, $f$ must be ITI there, so $\CC^2$-close to $\id$ on some small enough neighbourhood, which, according to Yoccoz, results in the $\CC^1$-smallness of $\xi_f$ there (see \ref{p:estb}, case $n=0$, for a precise statement and its proof). \medskip

As for $(iv)$, simply take $h=f^r\circ g^s$, where $(r,s)\in\Z^2$ satisfy $rq+sp=1$. 
\end{proof}

\subsubsection{Clean representations}
%
We want to prove the connectedness of $\Rs$, for any $\partial\subset\{0,1\}$. The idea is to find a subset of it which is both dense and path-connected. Given Corollary \ref{c2:kst}, $\RCp$ below is a natural candidate:

\begin{definition}\label{d:cleans} A representation $(f,g)\in\Rss$ is said to be \emph{clean} if $f$ and $g$ belong either to a common $\Cinf$ flow on $[0,1]$ or to a common infinite cyclic group generated by a $\Cinf$-diffeomorphism of $[0,1]$.
We denote by $\RCs\subset\Rss$ the subspace of clean representations and by $\RCp$ the subspace $\RCs\cap\Rs$ of $\Rs$.
\end{definition}
%

What we want to prove now is:

\begin{proposition}\label{t:tp} For any $\partial\subset\{0,1\}$,  $\RCp$ is path-connected and dense in $\Rs$. As a consequence, $\RCs$ is path-connected and dense in $\Rss$.
\end{proposition}

Path-connectedness is immediate (see ``Idea of the proof" in the introduction). Now if we define
$$\Rsr=\{(\id,\id)\}\cup \left\{(f,g)\in \Rs\setminus\{(\id,\id)\} : \tau_{g/f}\in\Q\cup\{\infty\}\right\}$$
and 
$$\Rsi=\{(\id,\id)\}\cup \{(f,g)\in \Rs \setminus\{(\id,\id)\}: \tau_{g/f}\in\R\setminus\Q\},$$
Proposition \ref{t:kst} directly implies:

\begin{corollary} \label{c:RC-R}
\begin{itemize}
\item  $\Rv=\RCv$, that is: every $\Z^2$-action on $[0,1]$ without ITI-fixed points is either monogeneous or is embedded in a smooth flow.  
\item and $\Rsr\subset\RCp$, that is: every $\Z^2$-action on $[0,1]$ without ITI-fixed points in $(0,1)$ and with rational translation number is monogeneous. 
\end{itemize}
\end{corollary}

So in order to get Proposition \ref{t:tp}, the only thing left to prove is:

\begin{proposition}\label{p:RC-Ri} $\RCsi:=\RCp\cap\Rsi$ is dense in $\Rsi$ for any nonempty $\partial\subset\{0,1\}$. 
\end{proposition}
This is the subject of the whole next section.

\subsection{Density of clean representations in the irrational translation number case}\label{ss:approx}

From now on, for any $f\in\DD$, $i\in\N$ and $x\in[0,1]$, $f^{\pm i}(x)$ and $f^{\mp i}(x)$ will denote $\max\left(f^i(x),f^{-i}(x)\right)$ and $\min\left(f^i(x),f^{-i}(x)\right)$ respectively.

\subsubsection{Approximation result}\label{sss:approx}
We obtain proposition \ref{p:RC-Ri} as a consequence of the following:

\begin{proposition}\label{p:pp}
Let $\{0\}\subset\partial\subset\{0,1\}$ and $(f,g)\in\Rsi\setminus\{(\id,\id)\}$ such that $|\tau_{g/f}|<1$. For all $\eps>0$, $a\in(0,1]$  and $k\in\N$, there exists $x_0\in(0,a]$ and a vector field on $[0,1]$ coinciding with the global Szekeres vector field of $f$ on $[x_0,1]$, $\CC^\infty$ on $[0,1)$, infinitely flat at $0$,  and $\eps$-$\CC^k$-small on $[0,f^{\pm2}(x_0)]$.
\end{proposition}

Of course, there is a similar statement for the case $\{1\}\subset\partial\subset \{0,1\}$. 

In Proposition~\ref{p:pp}, we use $(f,g)\in\Rsi\setminus\{(\id,\id)\}$ only for ensuring the existence of a Szekeres flow for $f$ defined on $[0,1)$, even if $f$ admits infinitely many non-flat fixed points tending to $0$. Then Proposition~\ref{p:pp} performs a perturbation  smoothing the Szekeres flow at $0$. We think that such a statement may be useful even in other settings than $\Z^2$-actions.  Therefore, let us reformulate Proposition~\ref{p:pp} as follows: 

\begin{proposition}\label{p:ppbis}{\bf (Smoothing the Szekeres flow)}
Let $f$ be a smooth diffeomorphism of $[0,1)$, without $\infty$-flat fixed points in $(0,1)$ and such that $f$ is the time one map of the flow of a   $\CC^1$ vector field $\xi$ on $[0,1)$ (notice that $\xi$ is necessarily smooth on $(0,1)$). 

Then, for all $\eps>0$, $a\in(0,1]$  and $k\in\N$, there exists $x_0\in(0,a]$ and a vector field on $[0,1)$ coinciding with $\xi$ on $[x_0,1]$, $\CC^\infty$ on $[0,1)$, infinitely flat at $0$,  and $\eps$-$\CC^k$-small on $[0,f^{\pm2}(x_0)]$.
\end{proposition}


\begin{remark} The method we use to prove \ref{p:pp} (cf. \ref{ss:recol}) does not necessarily produce a vector field \emph{without infinitely flat zeros in} $(0,x_0)$. This must be taken into account in the proof of \ref{p:RC-Ri} below, where we want to build a vector field nowhere infinitely flat in $(0,1)$.
\end{remark}

\begin{proof}[Proof of \ref{p:RC-Ri}]
We restrict to the case $\partial=\{0\}$, the cases $\partial=\{1\}$ and $\partial=\{0,1\}$ being similar. Let $(f,g)\in \Rsi\setminus\{(\id,\id)\}$. Permuting $f$ and $g$ if necessary, we can assume that $|\tau_{g/f}|<1$. Denote by $\xi$ the  global Szekeres vector field of $f$ on $[0,1]$ (cf. \ref{t:kst}$(iii)$). Let $\eta>0$ and $k\in\N$. We want to build a pair $(\ft,\gt)\in\RCsi$ $\eta$-$\CC^k$-close to $(f,g)$. This is done in two steps: first we construct an approximation $(\fb,\gb)$ belonging to a smooth flow, and then we modify it slightly to suppress the potential infinitely flat fixed points in $(0,1)$.

Since $f$ and $g$ are ITI at $0$, there exists $a>0$ such that
\begin{equation}\label{e:fbgb}
\norm{f-\id}_{k,[0,f^{\pm1}(a)]}\le \frac\eta4\quad\text{and}\quad \norm{g-\id}_{k,[0,f^{\pm1}(a)]}\le \frac\eta4.
\end{equation}
Now according to Proposition \ref{p:pp}, given any $\eps>0$, there exists $x_0\in(0,a]$ and a $\Cinf$ vector field $\xib$ on $[0,1]$ infinitely flat at $0$ and such that
\begin{equation*}\label{e:xib}
\xib\res{[x_0,1]}=\xi\res{[x_0,1]}\quad\text{and}\quad \norm{\xib}_{k,[0,f^{\pm2}(x_0)]}\le\eps.
\end{equation*}
As a consequence, for all $|t|\le 1$, the time-$t$ map of $\xib$ coincides with that of $\xi$ on $[f^{\pm1}(x_0),1]$ and, provided $\eps>0$ is chosen small enough, is $\frac{\eta}4$-$\CC^k$-close to $\id$ on $[0,f^{\pm1}(x_0)]$. In particular, if $\fb$ and $\gb$ denote the time-$1$ and $\tau_{g/f}$ maps of $\xib$, 
\begin{equation*}
\norm{\fb-\id}_{k,[0,f^{\pm1}(x_0)]}\le \frac\eta4,\quad\norm{\gb-\id}_{k,[0,f^{\pm1}(x_0)]}\le \frac\eta4,
\end{equation*}
\begin{equation*}
\fb\res{[f^{\pm1}(x_0),1]}=f\res{[f^{\pm1}(x_0),1]}\quad\text{and}\quad\gb\res{[f^{\pm1}(x_0),1]}=g\res{[f^{\pm1}(x_0),1]}.
\end{equation*}
So in the end, given \eqref{e:fbgb},
$$\norm{f-\fb}_{k,[0,1]}\le \frac\eta2\quad\text{and}\quad \norm{g-\gb}_{k,[0,1]}\le \frac\eta2.$$
But $(\fb,\gb)$ belongs to $\RCsi$ only if $\xib$ is nowhere infinitely flat outside of $\{0\}$, which is not guaranteed by Proposition \ref{p:pp}. Assume $\xib$ does have infinitely flat fixed points in $(0,1)$. Those necessarily belong to $(0,x_0)$ for $\xib$ coincides with $\xi$ on $[x_0,1]$ and $\xi$ is nowhere infinitely flat there by definition of $\Rs$. 

Let $b<x_0\le a$ be the biggest of them, let
\begin{align*}h_b : [0,1]&\to [b,1]\\
y\hspace{0.3cm}&\mapsto (1-b)y+b
\end{align*}
and define
$$(\ft,\gt)=(h_b^{-1}\circ\fb\circ h_b,h_b^{-1}\circ\gb\circ h_b).$$
This time, $(\ft,\gt)$ belongs to $\RCsi$ and, provided $a$ has been chosen small enough at the beginning, $h_b$ ($b<a$) is $\CC^k$-close enough to the identity for $(\ft,\gt)$ to be $\frac{\eta}2$-$\CC^k$-close to $(\fb,\gb)$, and thus $\eta$-$\CC^k$-close to $(f,g)$.
\end{proof}

The proof of Proposition \ref{p:pp} consists of two steps: first (cf. \ref{sss:control} below) obtain bounds on the derivatives of the global Szekeres vector field of $f$ in some specific disjoint regions closer and closer to $0$ (while no such bounds exist, in general, on whole neighbourhoods of $0$), and then interpolate between such regions and $0$ to replace $\xi_f$ there by some $\Cinf$ (and $\Cinf$-small) vector field (cf. \ref{ss:recol}).

\subsubsection{Local control on the derivatives of a global Szekeres vector field}\label{sss:control}

Proposition \ref{p:est} below, which is the key to Proposition \ref{p:pp} (cf \ref{ss:recol}), claims that, though the global Szekeres vector field of an element of $\Dsp$ (when it is well-defined) is only $\CC^1$  near $\partial$ in general, arbitrarily close to $\partial$, there are regions where it is ``$\Cinf$-small". This statement and its proof are widely based on the proof by Sergeraert that, \emph{under some non-oscillation condition}, a $\Cinf$ diffeomorphism of $[0,1)$ without fixed point in $(0,1)$ and ITI at $0$ has a $\Cinf$ Szekeres vector field (cf. \cite{Se} Theorem 3.1). For our purpose, however, we need to extend Sergeraert's ideas to the case of diffeomorphisms of the closed interval with possibly infinitely many fixed points. 

\begin{proposition}\label{p:est} Let $f\in\Dsp$, $\partial\subset\{0,1\}$. Assume $f$ is the time-$1$ map of a vector field $\xi$ $\CC^1$ on $[0,1]$ (and necessarily $\Cinf$ on $(0,1)$ by \ref{t:Takens}). Then for all $\delta>0$, $k\in \N$ and $i\in\partial$, there exists $x_0\neq f(x_0)$ arbitrarily close to $i$ such that 
\begin{equation*}
\norm{\xi}_{k,[f^{\mp2}(x_0),f^{\pm 2}(x_0)]}\le |f(x_0)-x_0|^{1-\delta}.
\end{equation*}
\end{proposition}

The proof is a combination of the following lemmas (in increasing order of difficulty). While Lemma \ref{l:dist} is elementary (see proof below), we dedicate a whole section to the proofs of \ref{t:szek} and \ref{p:estb} at the end of the article (cf. \ref{s:est}). Let us stress again that such results were already known for diffeomorphisms $f$ of $[0,1)$ \emph{without fixed points} in $(0,1)$ and ITI at $0$ (see \cite{Se}2.9 and \cite{Se}3.6 for analogs of \ref{t:szek} and \ref{p:estb} respectively).
What we do in Section \ref{s:est} is check that Sergeraert's arguments also work for diffeomorphisms with (possibly) infinitely many non-ITI fixed points accumulating on $0$. Though no new idea is involved, this unfortunately requires a whole rewriting of Sergeraert's proofs, in more details. Indeed, one could be tempted to simply apply his results to $f\res{[a,b)}$ for every connected component $(a,b)$ of $[0,1]\setminus\Fix(f)$. But first of all, strictly speaking, the results we mentionned only apply to diffeomorphisms which are ITI at the boundary of the interval under consideration, which is not the case of $f\res{[a,b)}$ unless $a= 0$, and more importantly, we need to make sure that Sergeraert's estimates are \emph{uniform} in the sense that, in our setting, they can be made independent of the component of $[0,1]\setminus\Fix(f)$ we apply them to. What makes the adaptation even more tedious (though this is mainly a matter of notation) is that the expression of the Szekeres vector field of $f\res{[a,b)}$, for a given connected component $(a,b)$ of $[0,1]\setminus\Fix(f)$, depends on the sign of $f-\id$ on $(a,b)$. 

%

\begin{lemma}\label{l:dist} Let $f$ be \emph{any} $\CC^1$-diffeomorphism of $[0,1)$ satisfying $Df(0)=1$. Then
$$\sup_{y\in[x,f^{\pm2}(x)]}\left|\frac{f(y)-y}{f(x)-x}-1\right|\underset{{x\to 0}\atop{x\notin\Fix(f)}}{\to}0.$$
\end{lemma}

\begin{lemma}\label{t:szek} Let $f$ and $\xi$ be as in Proposition \ref{p:est}. Then
\begin{itemize} 
\item $\log\left|\frac{\xi}{f-\id}\right|$ is bounded on $[0,1]\setminus\Fix(f)$;
\item if $Df(0)-1=D^2f(0)=0$ (in particular if $0\in\partial$), 
$$\xi(x)\underset{x\to 0}{\sim}f(x)-x\underset{x\to 0}{\sim}x-f^{-1}(x).$$
\end{itemize}
\end{lemma}

\begin{lemma}\label{p:estb}Let $f$ and $\xi$ be as in Proposition \ref{p:est}, with $0\in\partial$. For all $n\in\N^*$ and all $\eta>0$, 
$$\xi^{n-1}(x)D^{n}\xi(x)\underset{\underset{x\neq0}{x\to 0}}{=}O\left( \norm{f-\id}_{0,[0,x]}^{n-\eta}\right).$$
\end{lemma}

\begin{proof}[Proof of \ref{p:est}]
Let $f$ and $\xi$ be as in \ref{p:est}, $\delta>0$, $k\in\N$, and assume $i=0\in\partial$, to fix ideas. According to Lemma \ref{p:estb}, there exist $C>0$ and $x_1>0$ such that, for all $x\in(0,x_1]$ and all $n\in[\![0,k]\!]$,
\begin{equation}\label{e:estb1}
|\xi^{n-1}(x)D^{n}\xi(x)|\le C\norm{f-\id}_{0,[0,x]}^{n-\frac\delta2}.
\end{equation}
 Now according to \ref{l:dist}, \ref{t:szek}, and more generally to the ITI-ness of $f$ at $0$, there exists $x_2\in(0,x_1]$ such that, for all $x\in[0,x_2]\setminus\Fix(f)$,
\begin{equation}\label{e:point1}
\left|\frac{f(x)-x}{\xi(x)}\right|\le 2,\quad \frac12\le\frac{f(y)-y}{f(x)-x}\le 2\quad\forall y\in[x,f^{\pm2}(x)],
\end{equation}
and
\begin{equation}\label{e:iti}
\norm{f-\id}_{0,[0,x]}^{\frac\delta2}\le \frac{1}{2^{3k}C}.
\end{equation}
Pick $x_0\in[0,f^{\mp2}(x_2)]$ (so that \eqref{e:estb1}, \eqref{e:point1} and \eqref{e:iti} hold for all $x\le f^{\pm 2}(x_0)\le x_2$) satisfying $|f(x_0)-x_0| = \norm{f-\id}_{0,[0,x_0]}$. In particular $f(x_0)\neq x_0$.
Then for all $x\in [f^{\mp2}(x_0),f^{\pm 2}(x_0)],$ 
\begin{equation}\label{e:xxo}
\norm{f-\id}_{0,[0,x]} \le\norm{f-\id}_{0,[x_0,f^{\pm 2}(x_0)]}
\end{equation}
so, for all $n\in[\![0,k]\!]$, 
\begin{align*}|D^n\xi(x)|&\le C\frac{\norm{f-\id}_{0,[0,x]}^{n-\frac\delta2}}{|\xi(x)|^{n-1}}\quad \text{by \eqref{e:estb1}}\\
&\le C\frac{\norm{f-\id}_{0,[x_0,f^{\pm 2}(x_0)]}^{n-\frac\delta2}}{|\xi(x)|^{n-1}}\quad \text{by \eqref{e:xxo}}\\
&\le C2^{n-\frac\delta2}\frac{\left|f(x_0)-x_0\right|^{n-\frac\delta2}}{|\xi(x)|^{n-1}}\quad\text{ according to \eqref{e:point1}}\\
&=C2^{n-\frac\delta2}\frac{\left|f(x_0)-x_0\right|^{n-\frac\delta2}}{\left|f(x_0)-x_0\right|^{n-1}}\times\frac{\left|f(x_0)-x_0\right|^{n-1}}{\left|f(x)-x\right|^{n-1}}\times \frac{\left|f(x)-x\right|^{n-1}}{|\xi(x)|^{n-1}}\\ 
&\le C 2^{n-\frac\delta2}\times \left|f(x_0)-x_0\right|^{1-\frac\delta2}\times 2^{n-1}\times 2^{n-1}\quad \text{ according to  \eqref{e:point1}}\\
&\le |f(x_0)-x_0|^{1-\delta} \text{ according to \eqref{e:iti}},\end{align*}
which concludes the proof of Proposition \ref{p:est}.
\end{proof}

\begin{proof}[Proof of Lemma \ref{l:dist}] Let $x\in(0,1)\setminus \Fix(f)$. 
Then for all $y\in[x,f^{\pm 2}(x)]$, 
\begin{align}\label{e:dist}
\left|\frac{f(y)-y}{f(x)-x}-1\right|&=\left|\frac{(f(y)-y)-(f(x)-x)}{f(x)-x}\right|\notag\\
&\le \sup_{[x,f^{\pm 2}(x)]}|D(f-\id)|\left|\frac{y-x}{f(x)-x}\right|\notag\\
&\le  \sup_{[x,f^{\pm 2}(x)]}|D(f-\id)|\left|\frac{f^{\pm 2}(x)-x}{f(x)-x}\right|.
\end{align}
If $x<f(x)$,
\begin{align*}
\left|\frac{f^{\pm 2}(x)-x}{f(x)-x}\right|=\frac{f^2(x)-f(x)}{f(x)-x}+1=Df(a_x)+1\quad\text{for some $a_x\in[x,f(x)]$}.
\end{align*}
If $f(x)<x$,
\begin{align*}
\left|\frac{f^{\pm 2}(x)-x}{f(x)-x}\right|&=\frac{f^{- 2}(x)-f^{- 1}(x)}{x-f(x)}+\frac{f^{- 1}(x)-x}{x-f(x)}\\
&=\frac{f^{- 2}(x)-f^{- 2}(f(x))}{x-f(x)}+\frac{f^{- 1}(x)-f^{-1}(f(x))}{x-f(x)}\\
&=Df^{-2}(b_x)+Df^{-1}(c_x)\quad\text{for some $b_x$ and $c_x\in[f(x),x]$}.
\end{align*}
Now $Df$, $Df^{-1}$ and $Df^{-2}$ are bounded on $[0,1]$ and since $f$ is $\CC^1$-tangent to $\id$ at $0$, 
$\sup_{[x,f^{\pm 2}(x)]}|D(f-\id)|$ goes to $0$ when $x$ goes to $0$, which concludes the proof, according to \eqref{e:dist}.
\end{proof}

\subsubsection{Interpolation (proof of Proposition \ref{p:pp} using \ref{p:est})}\label{ss:recol}
In this section, we prove that Proposition \ref{p:est} implies Proposition \ref{p:pp}. 
\medskip

Let $0\subset\partial\subset\{0,1\}$, $(f,g)\in\Rsi\setminus\{(\id,\id)\}$ with $|\tau_{g/f}|<1$ (so in particular $f\neq\id$), and let $\xi$ denote the global Szekeres vector field of $f$ on $[0,1]$. Let $\eps>0$, $k\in\N$ and $a\in(0,1]$. We first describe, given \emph{any} $x_0\in[0,a]\setminus \Fix(f)$, a specific way to construct a $\Cinf$ vector field $\xib$ on $[0,1]$ coinciding with $\xi$ on $[x_0,1]$ and infinitely flat at $0$. Then we prove that $x_0$ can be chosen so that $\xib$ is $\eps$-$\CC^k$-small on $[0,f^{\pm2}(x_0)]$.\medskip

\noindent\textbf{Step 1: construction of $\xib$ for any $x_0$.}
Let $x_0\in[0,a]\setminus \Fix(f)$.
To extend the vector field $\xi\res{[x_0,1]}$ to $[0,1]$, the idea is simply to ``stretch" $\xi\res{[f^{\mp1}(x_0),x_0]}$ into a vector field on $[0,x_0]$ and then  flatten it near $0$. 
%
More precisely, what we ``stretch" is $D^k\xi\res{[f^{\mp1}(x_0),x_0]}$, and then we integrate $k$ times to obtain our new vector field.\medskip

Let $\psi:[0,1]\to [f^{\mp1}(x_0),1]$ be a $\Cinf$ diffeomorphism coinciding with the identity on $[x_0,1]$, and $\rho : \R\to [0,1]$ a $\Cinf$ function vanishing on $\R_-$, increasing on $[0,1]$ and constant equal to $1$ on $[1,+\infty)$. We first extend $D^k\xi\res{[x_0,1]}$ into a $\Cinf$ function $\alpha_0$ on $[0,1]$ as follows:
$$\alpha_0(x) = D^k\xi\circ \psi(x)\quad \text{for all $x\in[0,1]$}.$$
Then, by induction on $1\le i\le k$, define
\begin{equation}\label{e:recalpha}\alpha_i(x) = D^{k-i}\xi(x_0) + \int_{x_0}^x \alpha_{i-1}(u)du\quad\text{for all $x\in[0,1]$}.
\end{equation}
Note that, for all $1\le i\le k$, $\alpha_i=D^{k-i}\xi$ on $[x_0,1]$.
Now define
$$\xib(x) = \rho\left(\frac{x}{x_0}\right)\alpha_k(x)\quad \text{for all $x\in[0,1]$}.$$
By construction, $\xib$ is $\Cinf$ on $[0,1]$, infinitely flat at $0$ and coincides with $\xi$ on $[x_0,1]$. \medskip

\noindent\textbf{Step 2: choice of $x_0$.}
Fix $\delta\in(0,1)$ and take $x_0$ as in Proposition \ref{p:est}, i.e such that
\begin{equation}\label{e:est}
\norm{\xi}_{k,[f^{\mp2}(x_0),f^{\pm 2}(x_0)]}\le |f(x_0)-x_0|^{1-\delta}.
\end{equation}
Remember that such an $x_0$ can be found arbitrarily close to $0$. Let us estimate the $\CC^k$-norm of $\xib$ on $[0,f^{\pm 2}(x_0)]$ for such an $x_0$. First of all, on $[x_0,f^{\pm2}(x_0)]$, $\xib$ coincides with $\xi$, so
\begin{equation}\label{e:est2}
\norm{\xib}_{k,[x_0,f^{\pm 2}(x_0)]}\le |f(x_0)-x_0|^{1-\delta}.
\end{equation}
Next, for all $x\in[0,x_0]$ and all $0\le l\le k$,
\begin{equation}\label{e:dkxib}
|D^l\xib(x) |= \left|\sum_{i=1}^l\binom{l}{i}\frac{D^{i}\rho\left(\frac{x}{x_0}\right)}{x_0^{i}}D^{k-i}\alpha_k(x)\right|
\le l!\norm{\rho}_l\sum_{i=1}^l\frac{|\alpha_i(x)|}{x_0^{i}}.
\end{equation}
Now for all such $x$, $\psi(x)$ belongs to $[f^{\mp1}(x_0),x_0]$, so given \eqref{e:est},
$$|\alpha_0(x)| = |D^k\xi\circ \psi(x)|\le|f(x_0)-x_0|^{1-\delta},$$
and by induction on $i$ between $0$ and $k$, using \eqref{e:recalpha} and \eqref{e:est},
\begin{equation}\label{e:ai}
|\alpha_i(x)| \le (i+1)|f(x_0)-x_0|^{1-\delta}.
\end{equation}
So for all $x\in [0,x_0]$, \eqref{e:dkxib} and \eqref{e:ai} give
\begin{equation*}
|D^l\xib(x) |\le  l!\norm{\rho}_l\sum_{i=1}^l(i+1)\frac{|f(x_0)-x_0|^{1-\delta}}{x_0^i},
\end{equation*}
and consequently
\begin{equation}\label{e:est3}
\norm{\xib}_{k,[0,x_0]}\le k(k+1)k!\norm{\rho}_k\frac{|f(x_0)-x_0|^{1-\delta}}{x_0^k}.
\end{equation}
Finally, $f$ is ITI at $0$ so for $x_0$ small enough, the right-hand sides of \eqref{e:est2} and \eqref{e:est3} are less than $\eps$, and thus
$$\norm{\xib}_{k,[0,f^{\pm2}(x_0)]}\le \eps,$$
which terminates the proof of Proposition \ref{p:pp}. 
\subsection{From $\Rss$ to $\RR$ (proof of Theorem A)}\label{ss:general}

Like for $\Rss$, the strategy to prove the connectedness of $\RR$ is to find a subset of it which is both dense and path-connected.
Given \ref{c:kst}, a natural candidate for a dense subset would be the set of representations $(f,g)\in\RR$ whose restriction to the closure of any connected component of $[0,1]\setminus \ITI(f)\cap\ITI(g)$ is clean (clean $\Z^2$-actions on a given segment being defined just like on $[0,1]$).
But bearing in mind that our candidate must also be path-connected, we make the additional requirement that these connected components are in finite number:
\begin{definition}\label{d:cleang} A representation $(f,g)\in\RR$ is said to be \emph{piecewise-clean} if
\begin{itemize}
\item $[0,1]\setminus \ITI(f)\cap\ITI(g)$ is a \emph{finite} union of intervals;
\item for any such interval $I$, $(f\res{\bar{I}},g\res{\bar{I}})$ is clean.
\end{itemize}
We denote by $\RC\subset \RR$ the subset of piecewise-clean representations.
\end{definition}

Theorem A is a corollary of the following, since $(\id,\id)$ belongs to $\RC$.

\begin{proposition}\label{p:RC}
$\RC$ is path-connected and dense in $\RR$.
\end{proposition}

We now need to check that this proposition follows from its nondegenerate analog \ref{t:tp}, or more precisely from the (trivial) generalization of the latter to any segment $J\subset[0,1]$, Proposition \ref{p:tpg} below, whose statement requires a (trivial) generalization of our previous notations. Let 
\begin{itemize}
\item $\RJ = \{(f,g)\in(\DJ)^2 : f\circ g = g\circ f\}\subset (\DJ)^2,$ 
\item $\DJs$ be the subset of $\DJ$ made of the identity and the diffeomorphisms that are nowhere ITI in the interior of $J$,
\item $\RJs=\RJ\cap(\DJs)^2$, which is the same (cf. \eqref{e:Rs}) as $(\id,\id)$ together with the pairs of diffeomorphisms which are nowhere simultaneously ITI in the interior of $J$ (we call such pairs \emph{nondegenerate}),
\item $\RCJs\subset\RJs$ be the set of pairs which belong either to a common $\Cinf$ flow on $J$ or to a common infinite cyclic group generated by a $\Cinf$ diffeomorphism of $J$  (we call such pairs \emph{clean}),
\end{itemize}
and for any subset $\partial\subset \partial J$, let
\begin{itemize}
\item $\DJsp$ be the subset of $\DJs$ made of the identity and the diffeomorphisms that are ITI exactly at $\partial$ and nowhere else, 
\item $\RJsp=\RJs\cap(\DJsp)^2$,
\item $\RCJsp=\RCJs\cap(\DJsp)^2$.
\end{itemize}


\begin{proposition}\label{p:tpg}
For any $\partial\subset\partial J$, $\RCJsp$ is path-connected and dense in $\RJsp$. 
\end{proposition}

%

\begin{proof}[Proof of \ref{p:RC}] The path-connectedness of $\RC$ is an easy consequence of that of $\RCJsp$ (cf. \ref{p:tpg}) for any segment $J\subset[0,1]$ and any $\partial\subset\partial J$. Indeed, one can connect any element $(f,g)$ of $\RC$ to $(\id,\id)$ proceeding on the closure $\bar{I}$ of each connected component $I$ of $[0,1]\setminus  \ITI(f)\cap\ITI(g)$ independently and gluing everything back together smoothly since the gluing points are in finite number and at each of them, all the diffeomorphisms involved are ITI. \medskip

So let us now focus on the density of $\RC$ in $\RR$.
Let $(f,g)\in\RR$. We want to prove that for all $\eps>0$ and all $k\in\N$, there exists $(\ft,\gt)\in \RC$ such that 
$$\|f-\ft\|_k\le\eps\quad\text{and}\quad\norm{g-\gt}_k\le\eps.$$

Let $\Omega=[0,1]\setminus \ITI(f)\cap\ITI(g)$, $K$ be the union of the closures of all the connected components $I$ of $\Omega$ satisfying
$$\|f-\id\|_{k,\bar{I}}\le\eps\quad\text{and}\quad\|g-\id\|_{k,\bar{I}}\le\eps,$$
and let $\Omega'=\Omega\setminus K$.

\begin{claim} $\Omega'$ has finitely many connected components.
\end{claim}

\begin{proof}
The endpoints of the connected components of $\Omega'$ belong to $\{0,1\}\cup (\ITI(f)\cap\ITI(g))$. If they are infinitely many, they accumulate on a point, which  necessarily belongs to $\ITI(f)\cap\ITI(g)$. Then $f$ and $g$ are $\eps$-$\CC^k$-close to $\id$ on some neighbourhood of this point. But such a neighbourhood contains (infinitely many) connected components of $\Omega'=\Omega\setminus K$, which is incompatible with the definition of $K$. 
\end{proof}

We can now define $(\ft,\gt)$:

\begin{itemize}
\item on $K\cup (\ITI(f)\cap\ITI(g))$, set $\ft=\gt=\id$; 
\item by definition of $\Omega'$, for each connected component $I$ of $\Omega'$, $(f\res{\bar{I}},g\res{\bar{I}})$ belongs to $\RIbs$, with $\partial=\partial I\cap  \ITI(f)\cap\ITI(g)$. So according to \ref{p:tpg}, $(f\res{\bar{I}},g\res{\bar{I}})$ is $\eps$-$\CC^k$-close to some pair in $\RCIbs$. Define $(\ft\res{\bar{I}},\gt\res{\bar{I}})$ as such a pair.
\end{itemize}

The resulting maps $\ft$, $\gt$ are as closed as required from $f$ and $g$, and they are smooth since the diffeomorphisms we glue together to build them are ITI at the ``gluing points" ($\partial \Omega'$), which are finitely many. 

\begin{remark}\label{r:ITI}
It follows directly from this proof that if $f$ and $g$ are ITI at $0$ and $1$, the pairs $(\ft,\gt)$ approaching $(f,g)$ and the paths connecting them to $(\id,\id)$ also are. This will be useful in part \ref{s:circle} when we extend our connectedness result to actions on the circle.
\end{remark}
\end{proof}


\subsection{$\Z^n$-actions, for $n>2$}\label{ss:Zn}
Just like for $n=2$ (cf. \ref{ss:general}), the connectedness of the space $\RRn$, $n>2$, of representations of $\Z^n$ into $\DD$ follows from that of the space of \emph{nondegenerate} representations, meaning representations without common ITI fixed point in $(0,1)$.

 Given a subset $\partial$ of $\{0,1\}$, let $\RRns$ denote the space of representations whose only common ITI fixed points are the points of $\partial$. Recall that for us, a representation is nothing but a $n$-tuple $(f_1,...,f_n)$ of commuting diffeomorphisms. Like for $n=2$ (cf. \ref{t:kst} and \ref{c2:kst}), one can show that $f_1$,...,$f_n$ belong either to a common infinite cyclic group generated by a $\Cinf$ diffeomorphism of $[0,1]$ (case $\Rsr$ for $n=2$), or to a common $\CC^1$ flow on $[0,1]$ (case $\Rsi$ for $n=2$). In the first case, $(f_1,...,f_n)$ belongs to the path-connected component of the trivial representation in $\RRns$. And in the second case, one can show just like for $n=2$ (cf. \ref{p:RC-Ri}) that $(f_1,...,f_n)$ is approached by ``clean" n-tuples, that is n-tuples of diffeomorphisms which belong to a common $\Cinf$ flow on $[0,1]$. Indeed, the number of commuting diffeomorphisms at stake plays no role in the proof of \ref{p:RC-Ri}, since this proof mainly consists in studying the global Szekeres vector field of \emph{one} of the diffeomorphisms. Now clean $n$-tuples belong to the path-connected component of the trivial representation in $\RRns$, which completes the proof of the connectedness of $\RRns$.

%
%
%
%

\section{$\Z^n$ actions on the circle $\T^1$}\label{s:circle}
The aim of this section is the proof of Theorem~\ref{t:b}, that is, the connectedness of the space of $\Z^n$ orientation preserving actions on the circle $\T^1$ without free sub $\Z^2$-action; more precisely the path-connected component of the trivial action is dense in $\RSnnl$. We proceed by induction. Observe that the case $n=1$ is easy and classical. Thus we assume now that  Theorem~\ref{t:b} has been proved for $1,\dots, n-1$, and we will prove it for $n>1$.

Let $F\colon \Z^n\to \DS$ be a group morphism, and hence a $\Z^n$ action on $\T^1$. We assume that this action has no free sub $\Z^2$-action.
We need to show that $F$ admits arbitrarily $\Cinf$-small perturbations which are isotopic to the trivial action. 

\subsection{The non-injective case}
First assume that $F$ is not injective. Then the quotient $\Z^n/\ker F$ is an abelian group. As a consequence $\Z^n/\ker F$ is a direct sum (abelian product)  $G\times T$ where $T$ is a finite group and $G$ is isomorphic to $\Z^m$,  $m\in\{0,\dots,n-1\}$. 

Any finite group of $\DS$ is smoothly conjugated to a rotation group. Thus, up to a smooth conjugacy, one may assume that $F(T)$ is a finite rotation group.  The quotient of $\T^1$ by $F(T)$ is a smooth circle, we denote it by $\T_T$.  As $F$ is an abelian action, every diffeomorphism in $F(\Z^n)$ commutes with $F(T)$ and therefore passes to the quotient in a diffeomorphism of $\T_T$. 

Let $F_T\colon \Z^m\to \Diff^\infty_+(\T_T)$ be the induced $G\simeq \Z^m$ action. One easily checks the following statement:
\begin{lemma}\label{l.lift} Every perturbation of $F_T$ can be lifted in a perturbation of $F|_G$ which commutes with $F(T)$, and hence defines a perturbation of $F$. 
\end{lemma}
\begin{proof}[Idea of the proof] Every generator of the perturbed action on the quotient can be lifted in a perturbation of the corresponding generator of $F(G)$. The commutator of two of these lifts is a diffeomorphism close to the identity map, and whose projection on $\T_T$ is the identity. One deduces that this commutator is the identity. 
\end{proof}

\begin{remark} $F_T$ does not admit any free sub $\Z^2$-action. 
\end{remark}

Now the induction hypothesis implies that $F_T$ admits $\Cinf$-small perturbations on the path-connected component of the trivial action. According to Lemma~\ref{l.lift} this means that $F$ admits a $\Cinf$-small perturbation $\tilde F$ isotopic to a $\Z^n$-action which projects on $\T_T$ as the trivial action.  In other words, $\tilde F$ is isotopic to an action by rational rotations, hence is isotopic to the trivial action. 

This ends the proof of Theorem~\ref{t:b} in the case where $F$ is not injective. From now on we assume that $F$ is injective. 


\subsection{Infinite rotation number group}

The rotation number $\rho$ is a classical invariant built by Poincar\'e for orientation preserving homeomorphisms of the circle. It is not a group morphism on $\DS$.  However, it defines a morphism in restriction to any abelian group. In particular it induces a morphism $\rho\colon F(\Z^n)\to \T^1$. In this section, we show that if $F$ is injective, $\rho(F(\Z^n))$ cannot be infinite. 

Assume $\rho(F(\Z^n))$ is infinite. Then one of the elements  $f$ of $F(\Z^n)$ has an irrational rotation number. In particular, Denjoy theorem implies that $f$ is topologically conjugate to an irrational rotation. 

Now $F$ does not admit any free sub $\Z^2$-action.  Therefore, the kernel of $\rho$ is isomorphic to $\Z^{n-1}$.  Every element $g$ in this kernel has at least one fixed point $x_0$.  But since $f$ commutes with $g$, the orbit $f^i(x_0)$, $i\in\Z$, consists of fixed points of $g$. This orbit is dense, so $g$ is the identity map, which contradicts the injectivity hypothesis on $F$ (recall $n>1$). 


\subsection{Finite rotation number group: the two cases}

We assume, from now on, that the image of $F(\Z^n)$ by the rotation number homomorphism $\rho$ is finite. Let $G\subset \Z^n$ be the Kernel of $\rho\circ F$. It is a sub-lattice isomorphic to $\Z^n$ and the quotient $\Z^n/G$ is finite. 

The image $\rho(F(\Z^n))$ is a finite subgroup of $\T^1$ hence is isomorphic to some $\Z/k\Z$: in other words  $\rho(F(\Z^n))$ is $\{0,\frac 1k,\dots,\frac{k-1}k\}$. 

\begin{lemma}\label{l.base} There is a basis $(f,g_2,\dots,g_{n})$ of $\Z^n$ with $g_i\in G$  and such that $\rho(F (f))$ is a generator of  $\rho(F(\Z^n))$.  In particular, if $\rho(F(\Z^n))=\Z/k\Z$ with $k\neq 1$ that means: $\rho(F(f))= \frac \alpha k$ with $\alpha \wedge k=1$
\end{lemma}
\begin{proof} If $k=1$, that is $\rho(F(\Z^n))=\{0\}$, there is noting to do. Let us assume now $k>1$.

Let $h\in\Z^n$ be an element such that $\rho(F(h))=\frac 1k$.  Then $\R.h\cap \Z^n$ (here we consider $\Z^n$ as a subgroup of $\R^n$) is isomorphic to a discrete subgroup of $\R$, hence is an infinite cyclic group.  We set $f$ as a generator of this group. The group generated by $\rho(F(f))$ contains $\frac 1k$ hence is $\rho(F(\Z^n))$.  Fix some $\beta\in\N$ such that $\alpha \beta \equiv 1 [k]$.

Since $f$ is not a multiple, it can be completed into a basis $(f,h_2,\dots,h_{n})$ of $\Z^n$.  Set $g_i= h_i +{n_i}f$ with $n_i= -\beta k\rho(F(h_i))$.
\end{proof}

Consider  now a fixed point $x_0$ of  an element $h\in G$ such that $h$ does not coincide with the identity map in the neighborhood of $x_0$, that is, $x_0$ does not belong to the interior of $\Fix(h)$. Then, as a consequence of Kopell's lemma,  $x_0$ is a fixed point of every $g\in G$. Furthermore, the contact order with the identity at $x_0$ is the same for every $g$, unless  $g$ is  the identity map in the neighborhood of $x_0$.  As a consequence, one gets:

\begin{lemma} If $x_0$ is a $\Cinf$-flat non interior fixed point for some $h\in G$, then $x_0$ is a flat fixed point for every $g\in G$. 
\end{lemma}

Notice that, if $h\neq id $ has a flat fixed point, it also has a non-iterior flat fixed point. So the two cases \emph{$G$ with flat fixed points} and \emph{$G$ without flat fixed point} are well-defined. 

The existence or non-existence of flat fixed points lead to different arguments.

\subsection{Existence of flat fixed points} 
In this section we assume that $x_0$ is a common flat fixed point for $G$.  
\begin{remark} Since $\Z^n/G$ is finite, one easily checks that the orbit of $x_0$ under $\Z^n$ is finite: just notice that every element of $\Z^n$ can be written as $g_i\circ g$ where $g\in G$ and $g_1,\dots, g_k$ are representatives of the elements of $\Z^n/G$.  
\end{remark}  

Let $x_0,x_1,\dots, x_{k-1}$ be the the orbit of $x_0$ under $\Z^n$ endowed with an indexation which is cyclically ordered on the circle $\T^1$.  We denote by $I_i$ the oriented segment $[x_i,x_{i+1}]$, $i\in\Z/k\Z$. Any $h\in \Z^n$ acts on this orbit by $x_i\mapsto x_{i+j}$ where $\frac jk$ is the rotation number $\rho(h)$, and thus $h(I_i)=I_{i+j}$. 

Let $(f,g_2,\dots, g_{n})$ be the basis of $\Z^n$ given by Lemma~\ref{l.base}. Then $f(x_i)=x_{i+\alpha}$ and $f(I_i)=I_{i+\alpha}$ for every $i$. Notice that the first return of $f$ in a segment $I_i$ is $f^k$ and admits $\partial I_i$ as flat fixed  points. 

\begin{lemma}
Under the hypothesis above, $(f,g_2,...,g_n)$ admits an arbitrarilly $\CC^\infty$-small perturbation isotopic to the trivial action $(\id,...\id)$.
\end{lemma}

\begin{proof}
According to Theorem~\ref{t:b} there is an arbitrarilly  $\Cinf$-small perturbation  $(h_1,\dots h_n)$ of the restriction of $(f^k,g_2,\dots,g_{n})$ to $I_0$ and isotopies $h_{1,t},\dots,h_{k,t}$ with the following properties:
\begin{itemize}
\item the $h_{i,t}$ are all  flat at $\partial I_0$,
\item for every $t$ the $h_{i,t}$, $i\in\{1,\dots, n\}$ define a $\Z^n$ action on $I_0$,
\item $h_{i,0}=h_i$ for every $i$,
\item $h_{i,1}=id$ for every $i$.
\end{itemize}
Now define $f_t$ as follows:
\begin{itemize}
\item for $i\neq -\alpha$, the restriction of $f_t$ to $I_i$ is $f$.
\item the restriction of $f_t$ to $I_{-\alpha}$ is $h_{1,t}\circ f^{1-k}$ 
\end{itemize}

Notice that the restriction of $f_t^k$ to $I_0$ is $f_t|_{I_{-\alpha}}\circ f^{k-1}|_{I_0}= h_{1,t}$, and therefore commutes with $h_{j,t}$ for every $j\in\{2,\dots, n\}$.

For every such $j$, we define $g_{j,t}$ on $I_{i\alpha}$, $i\in \Z$, by $g_{j,t}=f_t^i\circ h_{j,t}\circ f_t^{-i}$. This is well-defined because $f_t^k|_{I_0}$ commutes with $h_{j,t}$ so that $ f_t^i\circ h_{j,t}\circ f_t^{-i}$ only depends on the class of $i$ in $\Z/k\Z$. For every $i$, the restriction of $g_{j,t}$ to $I_{i\alpha}$ is flat on the boundary, so the global map $g_{j,t}$ is a \emph{smooth} diffeomorphism of the circle $\T^1$. 

One easily checks that $(f_t,g_{2,t}\dots,g_{n,t})$ realises an isotopy of actions $F_t, t\in[0,1]$,  of $\Z^n$ on $\T^1$ such that $F_0$ is an (arbitrarilly) $\Cinf$-small perturbation of $F$ and $F_1$ is $(f_1,id,id,\dots,id)$. Finally, $F_1$ is isotopic to the trivial action, ending the proof. 

\end{proof}

\subsection{No flat fixed points: relative translation numbers... two cases}

We are now left with the case where no non-trivial element of $F(\Z^n)$ has an infinitely degenerate fixed point. However, every $g_i$ (cf. Lemma \ref{l.base}) has fixed points and all the non-trivial elements of $G$ have the same set of fixed points (cf. \ref{c:kst}). Let $x_0$ be a fixed point for the action of $G$.
 Notice that $(f^k,g_2,\dots,g_n)$ defines a nondegenerate $\Z^n$ action on the oriented segment $[x_0, x_0]$ obtained by cutting $\T^1$ at $x_0$. 
Therefore, since $f^k\neq \id$, the relative translation numbers $\tau_{g_i/f^k}$ are well-defined (cf. \ref{t:kst}).  Futhermore $\tau_{./f^k}$ is a morphism from $<f^k,g_2\dots g_n>\simeq \Z^n$ to $\R$. 

If this morphism is not injective, then two elements of $<f^k,g_2\dots g_n>$ have the same translation number with respect to $f^k$ and therefore coincide, which contradicts our injectivity hypothesis on $F$.

\subsection{No flat fixed point: irrational translation numbers}

We are left to consider the case where $\tau_{./f^k}$ is injective.  If $n>1$ this implies that the group $\tau_{./f^k}(<f^k,g_2\dots g_n>)$ contains an irrational number. An easy adaptation of \ref{t:kst} to the circle case gives:
  
\begin{lemma}
Under the hypotheses above, there is a unique $\Cinf$ vector field $X$ on $\T^1$ such that $f^k= X_1$ and $g_i= X_{t_i}$ for some $t_i\in \R$. 
Furthermore, $f$ leaves the vector field $X$ invariant, that is $f_*(X)=X$.
\end{lemma}

We conclude the proof by defining an isotopy of the action $F$ keeping $f$ unchanged by $g_{i,t}=X_{(1-t).t_i}$, for $t\in [0,1]$. 
For $t=1$ one gets an action generated by $f,id,\dots, id$ which is isotopic to the trivial action.

\section{Estimates for the Szekeres vector field}\label{s:est}

This section is devoted to the proofs of Lemma \ref{t:szek} and \ref{p:estb}. Let us stress again (cf. \ref{sss:control}) that, though long and rather technical, these proofs essentially repeat arguments by Sergeraert (cf.  \cite{Se}2.9 and \cite{Se}3.6), only in the more general setting of diffeomorphisms with possibly infinitely many fixed points.

Henceforth, $f$ is an element of $\Dsp$, $\partial\subset\{0,1\}$, whose global Szekeres vector field $\xi$ is well-defined. This means that for every connected component $(a,b)$ of $[0,1]\setminus\Fix(f)$, the Szekeres vector fields of $f\res{[a,b)}$ and $f\res{(a,b]}$ coincide on $(a,b)$, and $\xi$ is the vector field coinciding with these on each $(a,b)$ and vanishing on $\Fix(f)$. In particular, according to \ref{t:Takens}, $\xi$ is $\Cinf$ on $[0,1]\setminus\partial$. Now there are ``explicit" formulas for those (local) Szekeres vector fields and their derivatives, and thus for $\xi$ and $D\xi$ (cf. \eqref{e:xi} and \eqref{e:Dxi} below). Those are the starting point to prove estimates \ref{t:szek} and \ref{p:estb}. We refer the reader to \cite{Se}, \cite{Yo} or \cite{Na1}, among others, for a proof of these formulas. 

Let $\eta_0$ denote the $\Cinf$ vector field on $[0,1]$ defined by $\eta_0(x) = (f(x)-x)\partial x$. Recall that for all $k\in\N$ and $x\in[0,1]$, $f^{\pm k}(x)$ and $f^{\mp k}(x)$ denote $\max\left(f^k(x),f^{-k}(x)\right)$ and $\min\left(f^k(x),f^{-k}(x)\right)$ respectively. Note that, on every connected component of $[0,1]\setminus \Fix(f)$, $f^{\pm k}$ coincides with either $f^k$ or $f^{-k}$, and thus induces a $\Cinf$ diffeomorphism. Then, for every connected component $(a,b)$ of $[0,1]\setminus \Fix(f)$, on $[a,b)$,
\begin{equation}\label{e:xi}
\xi=\tau_a\lim_{k\to +\infty}(f^{\mp k})^*\eta_0, \quad\text{with}\quad \tau_a=\begin{cases}\frac{\log Df(a)}{Df(a) - 1}&\text{if $Df(a) \neq 1$}\\
1&\text{otherwise,}\end{cases}
\end{equation}
and
\begin{equation}
\label{e:Dxi}D\xi =\log Df(a)-\sum_{i=0}^{+\infty}
    (Lf^{\mp1} \times \xi)\circ f^{\mp i}
\end{equation}
(part of the argument of Szekeres' theorem \ref{t:szek} consists in proving that the sequence and series above converge uniformly on every compact subset of $[a,b)$). 

\subsection{Proof of Lemma \ref{t:szek}}

We want to prove that $\log\left|\frac{\xi}{f-\id}\right|$ is bounded on $[0,1]\setminus \Fix(f)$ and tends to $0$ at $0$ if $f$ is $\CC^2$ tangent to $\id$ at $0$. Denote by $\eta_k$ the vector field on $[0,1]$ vanishing on $\Fix(f)$ and equal to $(f^{\mp k})^*\eta_0$ on every connected component of $[0,1]\setminus\Fix(f)$. According to \eqref{e:xi}, for all $x$ in such a component $(a,b)$,
\begin{equation}\label{e:equiv}\log\frac{\xi(x)}{f(x)-x}=\log\tau_a+\lim_{k\to+\infty}\log\frac{\eta_k(x)}{\eta_0(x)}=\log\tau_a+\sum_{k=0}^{+\infty}\log\frac{\eta_{k+1}(x)}{\eta_k(x)}.
\end{equation}
Recall that $\tau_a=\frac{\log Df(a)}{Df(a)-1}$ if $Df(a)\neq 1$, $1$ otherwise. So if $C_x$ denotes $\norm{\log Df}_{0,[0,x]}$, 
\begin{equation}\label{e:taua}
|\log\tau_a|\le \sup_{{|y|\le C_x}\atop {y\neq 0}}\log\frac{y}{e^y-1}=:M_x\le \sup_{{|y|\le C_1}\atop {y\neq 0}}\log\frac{y}{e^y-1}=M_1<+\infty.
\end{equation}
Moreover,
$$ \log \frac{ \eta_{k+1} }{ \eta_k } 
 = \log \frac{ (f^{\mp k})^* \eta_1 }{ (f^{\mp k})^* \eta_0 }  
 = \log \frac{ \eta_1 \circ f^{\mp k} }{ \eta_0 \circ f^{\mp k} }
 = \theta \circ f^{\mp k}, $$
where $\theta:=\log \frac{ \eta_1}{ \eta_0}$ on $[0,1]\setminus \Fix(f)$. Now for all $x\in[0,1]\setminus \Fix(f)$,

\begin{align*}|\theta (x)|
             &= \left|\log \frac{ f^{\mp2}(x) - f^{\mp1}(x) }{ Df^{\mp1}(x) \, (f^{\mp1}(x) - x) }\right|\\
&=\left|\log\frac{ f^{\mp2}(x) - f^{\mp1}(x) }{ (f^{\mp1}(x) - x) }-\log Df^{\mp1}(x)\right|\\
&=\left|\log Df^{\mp1}(x_0)-\log Df^{\mp1}(x)\right|\quad\text{for some $x_0\in[f^{\mp1}(x),x]$}\\
&\le \norm{D\log Df^{\mp1}}_{0,[x_0,x]}|x_0-x|\\
&\le \norm{D\log Df^{\mp1}}_{0,[0,x]}(x-f^{\mp1}(x)).
\end{align*}
So, again for all $x\in[0,1]\setminus \Fix(f)$, \eqref{e:equiv} implies
\begin{equation}\label{e:rapport}\left|\log\frac{\xi(x)}{f(x)-x}\right|
\le \underbrace{\sup_{a\in[0,x]\cap\Fix(f)}\left|\log \tau_a\right|}_{\le M_x\le M_1\text{ by \eqref{e:taua}}}
+\| \underbrace{D\log Df^{\mp1}}_{{=\frac{D^2f^{\mp1}}{Df^{\mp 1}}}\atop{\text{bounded on $[0,1]$}}}\|_{0,[0,x]}
\underbrace{\sum_{k=0}^{+\infty}(f^{\mp k}(x)-f^{\mp (k+1)}(x))}_{\le x\le 1}.\end{equation}
Thus $\log\left|\frac{\xi}{f-\id}\right|$ is bounded on $[0,1]\setminus \Fix(f)$.
Now if $f$ is $\CC^2$-tangent to $\id$ at $0$, 
$$C_x\underset{x\to 0}{\to}1\quad\text{so}\quad M_x\underset{x\to 0}{\to}0,$$
and 
$$\norm{\frac{D^2f^{\mp1}}{Df^{\mp 1}}}_{0,[0,x]}\underset{x\to 0}{\to}0.$$
So by \eqref{e:rapport},
$$\left|\log\frac{\xi(x)}{f(x)-x}\right|\underset{{x\to 0}\atop {x\notin\Fix(f)}}{\to}0,\quad \text{(and $\xi(x)=f(x)-x=0$ if $x\in\Fix(f)$)}$$
which means
$$\xi(x)\underset{x\to0}{\sim}f(x)-x.$$
Now $f^{-1}$ is also $\CC^2$-tangent to $\id$ at $0$ and also has a global Szekeres vector field on $[0,1]$, which is none but $-\xi$, so similarly,
$$-\xi(x)\underset{x\to0}{\sim}f^{-1}(x)-x,$$
which concludes the proof of Lemma \ref{t:szek}.


\subsection{Proof of Lemma \ref{p:estb}}\label{ss:estb}

Here again, $f$ is an element of $\Dsp$, $\partial\subset\{0,1\}$, whose global Szekeres vector field $\xi$ is well-defined. We assume in addition that $\partial$ contains $0$, i.e that $f$ is ITI at $0$. We want to prove that for all $n\in\N^*$, 
\begin{equation}\label{e:Dnxi}\tag{$E_n$}
\forall \eta>0,\quad\xi^{n-1}(x)D^{n}\xi(x)\underset{\underset{x\neq0}{x\to 0}}{=}O\left( \norm{f-\id}_{0,[0,x]}^{n-\eta}\right).
\end{equation}
This is done by induction.
\subsubsection{Preliminaries}
For $n=1$, \eqref{e:Dnxi} becomes:
$$\forall \eta>0,\quad D\xi\underset{x\to 0}{=}O\left( \norm{f-\id}_{0,[0,x]}^{1-\eta}\right).$$
We will see that this estimate follows from expression \eqref{e:Dxi} for $D\xi$.
To prove Lemma \ref{p:estb} in general, the idea is to establish (by induction) expressions similar to \eqref{e:Dxi} for higher derivatives of $\xi$ (cf. \ref{l:alg} below), which will in turn provide the wanted estimate \eqref{e:Dnxi} (again by induction). More precisely, we are going to prove that, for all $n\ge 1$, on $[0,1]\setminus\Fix(f)$,
\begin{equation}\label{e:recxi}
\xi^{n-1}D^{n}\xi=P_{n}(D\xi,\xi D^2\xi,...,\xi^{n-2}D^{n-1}\xi)+\sum_{i=0}^{+\infty}
    \varphi_{n}\circ f^{\mp i}\quad(+c_1\text{ if $n=1$})
\end{equation}
for some $P_{n}$ and $\varphi_{n}$ defined below, where $c_1:[0,1]\setminus\Fix(f)\to\R$ denotes the locally constant function equal to $\log Df(a)$ on every connected component $(a,b)$ of $[0,1]\setminus \Fix(f)$. Formula \eqref{e:Dxi} gives such an expression for $n=1$, with $P_1=0$ and $\varphi_1=-Lf^{\mp 1}\times \xi$ on $[0,1]\setminus\Fix(f)$.  Note that $\varphi_1$ is well-defined and $\Cinf$ since, again, on every connected component of $[0,1]\setminus\Fix(f)$, $f^{\mp1}$ coincides either with $f$ or with $f^{-1}$, and thus induces a $\Cinf$-diffeomorphism of that component. \medskip

It will prove handy, to carry out the induction, to use the Lie derivative along $\xi$ rather than the usual derivative (mainly because of relations \eqref{e:lie} below). We will denote it by $L_\xi$: for every differentiable function $\varphi$ on an open subset $U$ of $(0,1)$, $L_\xi \varphi = D\varphi.\xi $ on $U$. Remember (cf. theorem \ref{t:kst}) that $\xi$ is $\Cinf$ on 
$[0,1]\setminus \partial\supset(0,1)$, so for every $\CC^k$ function $\varphi$ on $U\subset(0,1)$, $(L_\xi)^j\varphi$, $0\le j\le k$, is well-defined and $\CC^{k-j}$ on $U$. Finally, for readability reasons, let us introduce the following notations:
$$\forall n\ge 1,\quad \mu_n=\xi^{n-1}D^{n}\xi,\quad  \Phi_n= (L_\xi)^{n-1}D\xi\quad\text{and}\quad\varphi_n=(L_\xi)^{n-1}\varphi_1,$$
the first two being defined on $(0,1)$ and the last one on $[0,1]\setminus\Fix(f)$.
With these notations, 
\begin{equation}\begin{cases}
\mu_1=\Phi_1\\
L_\xi f^{\mp i}= \xi\circ f^{\mp i}\label{e:lie}\quad \forall i\in\Z\\
\mu_{n+1}=L_\xi\mu_{n}-(n-1)\mu_1\mu_{n}\quad \forall n\ge 1,
\end{cases}\end{equation}
which, by induction, leads to the following lemma. 
In this statement (for $n=1$), a polynomial (function) in $0$ variables composed with a $0$-tuple of functions (resp. of monomials in one variable) is to be understood as a \emph{constant} (resp. a constant polynomial in $1$ variable). We adopt this (controversial) convention only to make the statement and its proof simpler.
\begin{lemme}
\label{l:alg} For all $n\ge 1$, 
$$\mu_n=\Phi_n-P_n(\mu_1,...,\mu_{n-1})\quad\text{on $(0,1)$}$$
and
$$\varphi_n=-\sum_{q=0}^{n-1} D^qLf^{\mp1}\times \xi^{q+1}\times Q_{n,q}(\mu_1,...,\mu_{n-1})\quad\text{on $[0,1]\setminus\Fix(f)$},$$
for some polynomials $P_n$ and $Q_{n,q}$ in $n-1$ variables, \emph{independent of $f$}, with nonnegative (integer) coefficients, satisfying
\begin{equation}\label{e:alg}
P_n(X,...,X^{n-1}) = \alpha_n X^{n}\quad\text{and}\quad Q_{n,q}(X,...,X^{n-1})=\beta_{n,q}X^{n-1-q}\tag{$*_n$}
\end{equation}
 for some $\alpha_n,\beta_{n,q}\in \N$. 
\end{lemme}

The proof is given in the next section \ref{ss:alg}. To prove Lemma \ref{p:estb}, we combine these inductive formulas to the following estimates proved by Sergeraert in \cite{Se} using Hadamard inequalities. This time, his proofs adapt without any change to our setting.

\begin{lemma}[cf. \cite{Se}{3.3}]\label{l:had}
Let $g\in\Cinf([0,1],\R)$, infinitely flat at $0$. Then, for all $n\in\N$ and all $\eta>0$, 
$$\norm{g}_{n,[0,x]}  \underset{x\to 0}{=}O(\norm{g}_{0,[0,x]}^{1-\eta}).$$
\end{lemma}

\begin{corollary}[cf. \cite{Se}{3.4}]\label{c:had}Let $f$ be a $\Cinf$ diffeomorphism of $[0,1]$ ITI at $0$. Then for all $n\in\N$ and all $\eta>0$, 
$$\norm{\log Df}_{n,[0,x]} \text{ and } \norm{\log Df^{-1}}_{n,[0,x]}  \underset{x\to 0}{=}O(\norm{f-\id}_{0,[0,x]}^{1-\eta}).$$
\end{corollary}

These estimates, first used to control $\mu_1=D\xi$ (i.e to prove $(E_1)$) and then injected in the formulas of Lemma \ref{l:alg}, result, by induction, in the following Lemma, which contains Lemma \ref{p:estb} (cf. $(i)_{n+1}$). This induction is carried out in sections \ref{ss:base} and \ref{ss:step}. We set $\mu_0=0$ and recall that $c_1:[0,1]\setminus\Fix(f)\to\R$ denotes the locally constant function equal to $\log Df(a)$ on every connected component $(a,b)$ of $[0,1]\setminus \Fix(f)$.

\begin{lemme}\label{l:estb}
 For all $n\in\N^*$, for all $\eta>0$, 
\begin{enumerate}[label=$(\roman*)_n$]
\item $\mu_{n-1}(x)\underset{x\to 0}{=}O\left( \norm{f-\id}_{0,[0,x]}^{n-1-\eta}\right)$;
\item$|\f_n(x)|\underset{{x\to 0}\atop{x\notin\Fix(f)}}{=}O\left( |\xi(x)|\times\norm{f-\id}_{0,[0,x]}^{n-\eta}\right)$;
\item $\Phi_n= \sum_{i= 0}^{+\infty} \f_n\circ f^{\mp i} \quad(+\,c_1\text{ if $n=1$})\;$ on $[0,1]\setminus \Fix(f)$;
\item $|\Phi_n(x)|\underset{{x\to 0}\atop{x\notin\Fix(f)}}{=}O\left( \norm{f-\id}_{0,[0,x]}^{n-\eta}\right)$.
\end{enumerate}
\end{lemme}

\subsubsection{Proof of Lemma \ref{l:alg}}\label{ss:alg}

\noindent\emph{Base case}. For $n=1$, the statement follows directly from the definitions of $\mu_1$, $\Phi_1$ and~$\varphi_1$:
$$ \mu_1= D\xi = \Phi_1+0,$$
$$\f_1 =- Lf^{\mp1}\times \xi= - D^0Lf^{\mp1}\times \xi^1\times  1,$$
and $(*_1)$ is satisfied for $\alpha_1=0$ and $\beta_{1,0}=1$.\medskip

\emph{Inductive step.} Assume the statement is true for some $n\in\N^*$ (let us stress that with our convention, what follows works for $n=1$ as well). Then 
\begin{align*}\mu_{n+1}&=L_\xi\mu_n-(n-1)\mu_1\mu_{n}\quad\text{by \eqref{e:lie}}\\
&=L_\xi(\Phi_n-P_n(\mu_1,...,\mu_{n-1}))-(n-1)\mu_1\mu_{n}\\
&=\Phi_{n+1}-\sum_{i=1}^{n-1}\frac{\partial P_n}{\partial x_i}(\mu_1,...,\mu_{n-1})L_\xi\mu_{i}-(n-1)\mu_1\mu_{n}
\\
&=\Phi_{n+1}-\sum_{i=1}^{n-1}\frac{\partial P_n}{\partial x_i}(\mu_1,...,\mu_{n-1})(\mu_{i+1}+(i-1)\mu_1 \mu_{i})-n\mu_1\mu_{n}\\
&=\Phi_{n+1}-P_{n+1}(\mu_1,...,\mu_n)
\end{align*}
with
$$P_{n+1}(X_1,...,X_n)=(n-1)X_1X_{n}+\sum_{i=1}^{n-1}\frac{\partial P_n}{\partial x_i}(X_1,...,X_{n-1})(X_{i+1}+(i-1)X_1 X_i).$$
In particular,
$$P_{n+1}(X,...,X^{n}) = (n-1)X^{n+1}+\sum_{i=1}^{n-1}iX^{i+1}\frac{\partial P_n}{\partial x_i}(X,...,X^{n-1}).$$
Now by the induction hypothesis, 
$$P_n(X,...,X^{n-1}) = \alpha_n X^{n},$$ 
which, after differentiation, gives
$$\sum_{i=1}^{n-1}iX^{i-1}\frac{\partial P_n}{\partial x_i}(X,...,X^{n-1}) = n\alpha_n X^{n-1}.$$
So
\begin{align*}P_{n+1}(X,...,X^{n}) =  (n-1)X^{n+1} +n\alpha_n X^{n+1}= \alpha_{n+1}X^{n+1}
\end{align*}
with $\alpha_{n+1}=n-1+n\alpha_n$.
\medskip

Obtaining the formula for $\f_{n+1}$ is not harder, but more tedious, so we strongly advise the reader against thoroughly reading what follows (which, again, works for $n=1$ as well):

\begin{align*}
\f_{n+1}=L_\xi\f_n &= - \sum_{q=0}^{n-1} L_\xi\left(D^qLf^{\mp1}\times \xi^{q+1}\times Q_{n,q}(\mu_1,...,\mu_{n-1})\right)\quad\text{by the induction hyp.}\\
&= - \sum_{q=0}^{n-1} \Biggl(L_\xi (D^{q}Lf^{\mp1})\cdot \xi^{q+1}\cdot Q_{n,q}(\mu_1,...,\mu_{n-1})\\
&\hspace{2cm}+ D^qLf^{\mp1}\cdot L_\xi\xi^{q+1}\cdot Q_{n,q}(\mu_1,...,\mu_{n-1})\\
&\hspace{3cm}\left.+ D^qLf^{\mp1}\cdot \xi^{q+1}\cdot\sum_{i=1}^{n-1}\frac{\partial Q_{n,q}}{\partial x_i}(\mu_1,...,\mu_{n-1})L_\xi\mu_i\right)\\
&= - \sum_{q=0}^{n-1} \Biggl(D^{q+1}Lf^{\mp1}\cdot \xi^{q+2}\cdot Q_{n,q}(\mu_1,...,\mu_{n-1})\\
&\hspace{2cm}+ D^qLf^{\mp1}\cdot(q+1)\mu_1 \cdot\xi^{q+1}\cdot Q_{n,q}(\mu_1,...,\mu_{n-1})\\
&\hspace{2.5cm}\left.+ D^qLf^{\mp1}\cdot \xi^{q+1}\cdot\sum_{i=1}^{n-1}\frac{\partial Q_{n,q}}{\partial x_i}(\mu_1,...,\mu_{n-1})((i-1)\mu_1\mu_i+\mu_{i+1})\right)\\
&= - \sum_{q=0}^{n} D^{q}Lf^{\mp1}\cdot \xi^{q+1}\cdot Q_{n+1,q}(\mu_1,...,\mu_{n})
\end{align*}
with, for $1\le q\le n-1$,
\begin{align*}Q_{n+1,q}(X_1,...,X_{n})=
Q_{n,q-1}&(X_1,...,X_{n-1})+(q+1)X_1Q_{n,q}(X_1,...,X_{n-1})\\
&+\sum_{i=1}^{n-1}\bigl((i-1)X_1X_i+X_{i+1}\bigr)\frac{\partial Q_{n,q}}{\partial x_i}(X_1,...,X_{n-1}),
\end{align*}
for $q=0$, 
$$Q_{n+1,0}(X_1,...,X_{n})=X_1Q_{n,0}(X_1,...,X_{n-1})+\sum_{i=1}^{n-1}\bigl((i-1)X_1X_i+X_{i+1}\bigr)\frac{\partial Q_{n,0}}{\partial x_i}(X_1,...,X_{n-1})$$
and for $q=n$,
$$Q_{n+1,n}(X_1,...,X_{n})=Q_{n,n-1}(X_1,...,X_{n-1}),$$
(which shows, in particular, that for all $n\in\N^*$, $Q_{n,n-1}$ is constant equal to $1$).\medskip
%
\medskip

Now for $1\le q\le n-1$,
\begin{align*}Q_{n+1,q}(X,...,X^{n})&=
Q_{n,q-1}(X,...,X^{n-1})+(q+1)XQ_{n,q}(X,...,X^{n-1})
\\
&\hspace{6cm}+\sum_{i=1}^{n-1}iX^{i+1}\frac{\partial Q_{n,q}}{\partial x_i}(X,...,X^{n-1})\\
&=\beta_{n,q-1}X^{n-q} + (q+1) \beta_{n,q}X^{n-q}+ X^2 (n-q-1)\beta_{n,q}X^{n-q-2},
\end{align*}
by the induction hypothesis, the equality 
$$\sum_{i=1}^{n-1}iX^{i-1}\frac{\partial Q_{n,q}}{\partial x_i}(X,...,X^{n-1}) = (n-q-1)\beta_{n,q}X^{n-q-2}$$
being once again obtained by differentiating the induction hypothesis on $Q_{n,q}$ :
$$Q_{n,q}(X,...,X^{n-1})=\beta_{n,q}X^{n-q-1}.$$
So
$$Q_{n+1,q}(X,...,X^{n})=\beta_{n+1,q} X^{n-q}\quad\text{with}\quad\beta_{n+1,q} = \beta_{n,q-1} + n \beta_{n,q} .$$
For $q=0$ and $q=n$, a similar argument gives
$$Q_{n+1,0}(X,...,X^{n})=\beta_{n+1,0} X^{n}\quad\text{with}\quad\beta_{n+1,0} = n \beta_{n,0} $$
and
$$Q_{n+1,n}(X,...,X^{n})=\beta_{n+1,n} X^{0}\quad\text{with}\quad\beta_{n+1,n} = \beta_{n,n-1} =1 .$$
%

\subsubsection{Proof of \ref{l:estb} by induction : base case, $n=1$}\label{ss:base}
Let $\eta>0$. \medskip

\noindent$(i)_1$ is straightforward since $\mu_0=0$.\bigskip

\noindent$(ii)_1$ For all $x\in [0,1]\setminus \Fix(f)$, 
$$|\f_1(x) |=| (Lf^{\mp1}\times\xi)(x) |\le |\xi(x)|\times\norm{Lf^{\mp1}}_{0,[0,x]}\underset{x\to 0}{=}O\left(|\xi(x)|\times\norm{f-\id}^{1-\eta}_{0,[0,x]}\right)$$
according to Corollary \ref{c:had}.\bigskip

\noindent$(iii)_1$ This is exactly formula \eqref{e:Dxi}.\bigskip

\noindent$(iv)_1$ According to $(ii)_1$ above and Lemma \ref{t:szek}, there exist $C>0$ and $x_1\in (0,1]$ such that, for all $x\le x_1$, $x\notin\Fix(f)$,
\begin{equation}\label{e:iio}
|\f_1(x) |\le C |\xi(x)|\times\norm{f-\id}^{1-\eta}_{0,[0,x]}\quad\text{and}\quad \left|\frac{\xi(x)}{f^{\mp1}(x)-x}\right|\le 2.
\end{equation}
For all such $x$, for all $i\in\N$, $f^{\mp i}(x)\le x\le x_1$ so 
\begin{align*}\label{e:iioi}
|\f_1(f^{\mp i}(x)) |&\le C |\xi(f^{\mp i}(x))|\times\norm{f-\id}^{1-\eta}_{0,[0,f^{\mp i}(x)]}\\
&\le 2C\; |(f^{\mp1}-\id)\circ f^{\mp i}(x)|\times\norm{f-\id}^{1-\eta}_{0,[0,x]}\\
&\le 2C \;\left(f^{\mp i}(x)-f^{\mp (i+1)}(x)\right)\times\norm{f-\id}^{1-\eta}_{0,[0,x]}
\end{align*}
As a consequence,
\begin{equation*}
\left|\sum_{i=0}^{+\infty}
    \f_1\circ f^{\mp i}(x)\right|\le 2C  \norm{f-\id}^{1-\eta}_{0,[0,x]}\underbrace{\sum_{i=0}^{+\infty}(f^{\mp i}(x)-f^{\mp (i+1)}(x))}_{\le 1}.
\end{equation*}
Furthermore, by definition of $c_1$,
$$|c_1(x)| \le \norm{\log Df}_{0,[0,x]}\underset{x\to 0}{=}O\left(\norm{f-\id}_{0,[0,x]}^{1-\eta}\right)\quad\text{according to Corollary \ref{c:had}}.$$
So in the end, given $(iii)_1$,
$$|\Phi_1(x)|\le |c_1(x)|+\left|\sum_{i=0}^{+\infty}
    \f_1\circ f^{\mp i}(x)\right|\underset{x\to 0}{=}O\left( \norm{f-\id}_{0,[0,x]}^{1-\eta}\right),\quad\text{i.e. $(iv)_1$}.$$

\subsubsection{Proof of Lemma \ref{l:estb}: inductive step}\label{ss:step}

Let $n\ge1$, and assume $(i)_q$ to $(iv)_q$ are satisfied for all $q\le n$ (for all $\eta>0$). Let $\delta>0$.\bigskip\medskip

\noindent$(i)_{n+1}$ : According to $(iv)_n$ and $(i)_{2\,\text{to}\,n}$ 
, there exist $C>0$ and $x_1\in(0,1]$ such that, for all $x\le x_1$, $x\notin\Fix(f)$,
\begin{equation}\label{e:phin}
|\Phi_n(x)|\le C\norm{f-\id}^{n-\delta}_{0,[0,x]}\quad\text{and}\quad |\mu_{k}(x)|\le C^{k} \norm{f-\id}^{k-\frac{k\delta}{n}}_{0,[0,x]}\quad \forall k\in[\![1,n-1]\!].
\end{equation}
For all such $x$,
\begin{align*}
|\mu_{n}(x)| &= |\Phi_n(x) - P_n(\mu_1(x),...,\mu_{n-1}(x))|\quad\text{according to Lemma \ref{l:alg},}\\
& \le |\Phi_n(x)| + P_n(|\mu_1(x)|,...,|\mu_{n-1}(x)|)\quad\text{since the coefficients of $P_n$ are positive,}\\
&\le C\norm{f-\id}^{n-\delta}_{0,[0,x]}+ P_n\left(C\norm{f-\id}^{1-\frac{\delta}{n}}_{0,[0,x]},...,\left(C\norm{f-\id}^{1-\frac{\delta}{n}}_{0,[0,x]}\right)^{n-1}\right)\;\text{by \eqref{e:phin}}\\
&\le C\norm{f-\id}^{n-\delta}_{0,[0,x]} + \alpha_{n}\left(C\norm{f-\id}^{1-\frac{\delta}{n}}_{0,[0,x]}\right)^{n}\;\text{according to Lemma \ref{l:alg}}\\
&\le (C+\alpha_nC^{n}) \norm{f-\id}^{n-\delta}_{0,[0,x]},
\end{align*}
and this extends to $x\in\Fix(f)\cap(0,x_1]$ by continuity, which proves $(i)_{n+1}$.
\bigskip

\noindent$(ii)_{n+1}$ : According to $(i)_{2\,\text{to}\,n+1}$, there exist $C>0$ and $x_1\in(0,1]$ such that, for all $x\in(0,x_1]$, 
\begin{equation}\label{e:mum}
|\mu_k(x)|\le C^k \norm{f-\id}^{k-\frac{k\delta}{2n}}_{0,[0,x]}\quad \forall k\in[\![1,n]\!].
\end{equation}
In particular, for all such $x$, for all $q\in[\![0,n]\!]$,
\begin{align*}
\left| Q_{n+1,q}(\mu_1,...,\mu_{n})(x)\right|&\le Q_{n+1,q}\left( |\mu_1(x)|,...,|\mu_{n}(x)|\right)\,\text{since the coef. of $Q_{n+1,q}$ are $\ge0$,}\\
&\le Q_{n+1,q}\left(C\norm{f-\id}^{1-\frac{\delta}{2n}}_{0,[0,x]},...,\left(C\norm{f-\id}^{1-\frac{\delta}{2n}}_{0,[0,x]}\right)^{n}\right)\,\text{by \eqref{e:mum}}\\
&=\beta_{n+1,q}\left(C\norm{f-\id}^{1-\frac{\delta}{2n}}_{0,[0,x]}\right)^{n-q}\;\text{by Lemma \ref{l:alg}}\\
&\underset{x\to 0}{=}O\left( \norm{f-\id}^{n-q-\frac\delta2}_{0,[0,x]}\right).
\end{align*}
Now according to Lemma \ref{l:alg}, for all $x\in[0,1]\setminus\Fix(f)$,
\begin{align*}\label{e:fin}
|\f_{n+1}(x)|&=\left|\sum_{q=0}^{n} D^qLf^{\mp1}(x)\times \xi^{q+1}(x)\times Q_{n+1,q}(\mu_1,...,\mu_{n})(x)\right|\notag\\
&\le |\xi(x)|\times \sum_{q=0}^{n}\underbrace{\left|D^qLf^{\mp1}(x) \right|}_{{\underset{0}{=}O\left(\norm{f-\id}^{1-\frac\delta2}_{0,[0,x]}\right)}\atop {\text{according to\ref{c:had}}}}\times\underbrace{| \xi^{q}(x)|}_{{\underset{ 0}{=}O\left(\norm{f-\id}^{q}_{0,[0,x]}\right)}\atop{\text{according to \ref{t:szek}}}}\times \underbrace{\left| Q_{n+1,q}(\mu_1,...,\mu_{n})(x)\right|.}_{{\underset{0}{=}O\left(\norm{f-\id}^{n-q-\frac\delta2}_{0,[0,x]}\right)}\atop{\text{as we just saw}}}
\end{align*}
So $\quad|\f_{n+1}(x)|\underset{{x\to 0}\atop{x\notin\Fix(f)}}{=}O\left( |\xi(x)|\times\norm{f-\id}^{n+1-\delta}_{0,[0,x]}\right)$,$\quad\text{which proves $(ii)_{n+1}$}$.\medskip

Note that, more generally:
\begin{claim}\label{c:fibound}
$\frac{\varphi_{n+1}}{\xi}$ is bounded on $[0,1]\setminus \Fix(f)$. 
\end{claim}
\begin{proof}
According to Lemma \ref{l:alg}, on $[0,1]\setminus \Fix(f)$,
\begin{equation*}\label{e:fibound}
\frac{\varphi_{n+1}}{\xi}=\sum_{q=0}^{n} D^qLf^{\mp1}\times \xi^{q}\times Q_{n+1,q}(\mu_1,...,\mu_{n}).
\end{equation*}
For all $q\in[\![0,n]\!]$, $D^qLf^{\mp1}$ is bounded on $[0,1]\setminus\Fix(f)$ by $\max\left(\norm{Lf}_{n,[0,1]},\norm{Lf^{- 1}}_{n,[0,1]}\right)$. Furthermore, $\xi$ is continuous and thus bounded on $[0,1]$. Finally, $\mu_1=D\xi, ..., \mu_n=$ $\xi^{n-1}D^n\xi$ are continuous on $(0,1)$ by \ref{t:kst} and extend continuously to $[0,1)$ by estimates $(i)_{2\,\text{to}\,n+1}$. Naturally, if $f$ is ITI at $1$, similar estimates hold near $1$. And if $f$ is not ITI at $1$, $\xi$ is $\Cinf$ on a neighbourhood of $1$, so $\mu_1$,...,$\mu_n$ are too. So in the end, $\mu_1$,...,$\mu_n$ extend continuously to $[0,1]$, so they are bounded on $[0,1]$, and so is $Q_{n+1,q}(\mu_1,...,\mu_{n})$. 
\end{proof}
\bigskip

\noindent$(iii)_{n+1}$ : Note that
$$L_\xi(\varphi_n\circ f^{\mp i})=D\varphi_n\circ f^{\mp i}\times L_\xi f^{\mp i}=D\varphi_n\circ f^{\mp i}\times\xi \circ f^{\mp i}=\varphi_{n+1}\circ f^{\mp i},$$
So what we want to prove is that the symbol interversion $(*)$ below is licit:
$$\Phi_{n+1}=L_\xi\Phi_n\underset{(iii)_{n}}{=}L_\xi\left(\sum_{i= 0}^{+\infty} \f_n\circ f^{\mp i} \right)\underset{(*)}{=}\sum_{i= 0}^{+\infty} L_\xi(\f_n\circ f^{\mp i})=\sum_{i= 0}^{+\infty}\f_{n+1}\circ f^{\mp i}.$$
To that end, it is sufficient to prove that the last series converges uniformly on every segment contained in $[0,1]\setminus\Fix(f)$. Let $J$ be such a segment, and $(a,b)$ the connected component of $[0,1]\setminus\Fix(f)$ containing it. Let $C$ and $C'$ denote $\norm{\frac{\xi}{f-\id}}_{0,[0,1]\setminus \Fix(f)}$ (cf.~\ref{t:szek}) and $\norm{\frac{\varphi_{n+1}}{\xi}}_{0,[0,1]\setminus \Fix(f)}$ (cf. Claim \ref{c:fibound}) respectively. 
For all $x\in J$, for all $i\in\N$, $f^{\mp }i(x)\le x$, so 
\begin{equation}\label{e:iini}
\left|\f_{n+1}(f^{\mp i}(x)) \right|\le C' |\xi(f^{\mp i}(x))|\le C'C\left(f^{\mp i}(x)-f^{\mp(i+1)}(x)\right).
\end{equation}
So since $\sum_{i\ge 0}(f^{\mp i}-f^{\mp(i+1)})$ converges uniformly on $J$ (towards $\id -a$), so does $\sum_{i\ge 0} \f_{n+1}\circ f^{\mp i}$ and this concludes the proof of $(iii)_{n+1}$.\bigskip

\noindent$(iv)_{n+1}$: According to $(ii)_{n+1}$ and Lemma \ref{t:szek}, there exist $C>0$ and $x_1\in (0,1]$ such that, for all $x\le x_1$, $x\notin\Fix(f)$,
\begin{equation}\label{e:iin}
|\f_{n+1}(x) |\le C (x-f^{\mp 1}(x))\times\norm{f-\id}^{n+1-\delta}_{0,[0,x]}.
\end{equation}
So for all such $x$,
\begin{align*}
|\Phi_{n+1}(x)|&\le  \sum_{i\ge 0} |\f_{n+1}\circ f^{\mp i}(x)|\quad\text{by $(iii)_{n+1}$ above,}\\
&\le C'\norm{f-\id}^{n+1-\eta}_{0,[0,x]}\underbrace{\sum_{i\ge 0}(f^{\mp i}(x)-f^{\mp(i+1)}(x))}_{\le 1}\quad\text{by \eqref{e:iin},}
\end{align*}
which concludes the proof of $(iv)_{n+1}$ and thus the proof of Lemma~\ref{p:estb}.


\section*{Acknowledgements}
We are grateful to S. Crovisier, T. Tsuboi and J.-C. Yoccoz, among others, for sharing their ideas and knowledge around the subject of this article, and especially to S. Gouezel for making us realize at some point that we actually \emph{had} the solution to our problem. The second author would also like to thank E. Giroux for his amazing guidance all along, through countless discussions, from the first glimpse of understanding to the very formulation of some results.
We are also grateful to the Institut de Math\'ematiques de Bourgogne for its hospitality, and to the CIRM for the stimulating setting in which a major part of this work was accomplished.

\vskip 5mm
\begin{tabular}{ll}
Christian Bonatti & H\'el\`ene Eynard-Bontemps
\\
\footnotesize{bonatti@u-bourgogne.fr} &  \footnotesize{heynardb@math.jussieu.fr}
\\
Institut de Math\'ematiques de Bourgogne &Institut de Math\'ematiques de Jussieu 
\\
CNRS - UMR 5584&CNRS - UMR 7586
\\
Universit\'e de Bourgogne& Universit\'e Pierre et Marie Curie
\\
9 av. A. Savary & 4 place  Jussieu
\\
21000  Dijon, France& 75252 Paris cedex, France

\end{tabular}


\begin{thebibliography}{MM}
\bibitem[Be]{Be} 
M.~\textsc{Benhenda} ---
\textit{Circle diffomorphisms: quasi-reducibility and commuting diffeomorphisms.} 
preprint HAL


\bibitem[BoFi]{BoFi} 
C.~\textsc{Bonatti}; S.\textsc{Firmo} ---
\textit{Feuilles compactes d'un feuilletage g\'en\'erique en codimension $1$.} 
Ann. Sci. \'Ecole Norm. Sup. \textbf{4} 27 (1994), no. 4, 407--462.\


\bibitem[Br]{Br}
M.~\textsc{Brunella} ---
\textit{Remarks on structurally stable proper foliations}.
Math. Proc. Cambridge Philos. Soc. \textbf{115} (1994), no. 1, 111--120. 

\bibitem[Ey1]{Ey1}
H.~\textsc{Eynard} ---
\textit{Sur deux questions connexes de connexit\'e concernant les feuilletages et
leurs holonomies}. Ph. D. dissertation. {\tt
http://tel.archives-ouvertes.fr/tel-00436304/fr/}.\

\bibitem[Ey2]{Ey2}
H.~\textsc{Eynard} ---
\textit{A connectedness result for commuting diffeomorphisms of the interval}. Ergodic Theory and Dynam. Systems, \textbf{31} (2011), no.4, 1183--1191.

\bibitem[Ey3]{Ey3}
H.~\textsc{Eynard} ---
\textit{On the centralizer of diffeomorphisms of the half-line}. Comment. Math. Helv., \textbf{86} (2011), no.2, 415--435.

\bibitem[Ey4]{Ey4}
H.~\textsc{Eynard-Bontemps} ---
\textit{On the homotopy type of the space of codimension one foliations on a closed 3-manifold}. In preparation.



\bibitem[F--K]{F-K}
B.~\textsc{Fayad} ; K.~\textsc{Khanin} --
\textit{Smooth linearization of commuting circle diffeomorphisms}. \
 Ann. of Math. (2) \textbf{170} (2009), no. 2, 961--980.

\bibitem[Ko]{Ko} 
N.~\textsc{Kopell} ---
\textit{Commuting diffeomorphisms}. \ 
In \textit{Global Analysis}, 
Proc. Sympos. Pure Math. XIV, Amer. Math. Soc. (1968), 165--184.

\bibitem[La]{La}
A.~\textsc{Larcanché} ---
\textit{Topologie locale des espaces de feuilletages en surfaces des variétés
fermées de dimension 3}. \
Comment. Math. Helvetici  \textbf{82} (2007), 385--411.


\bibitem[Ma1]{Ma1}
J.~\textsc{Mather} ---
\textit{Commutators of $\CC^r$-diffeomorphisms of the real line}. \
Preprint, version préliminaire de :

\bibitem[Ma2]{Ma}
J.~\textsc{Mather} ---
\textit{Commutators of diffeomorphisms}. \
Comment.Math. Helvetici  \textbf{49} (1974), 512--528.


\bibitem[Na1]{Na1}
A.~\textsc{Navas} ---
\textit{Groups of circle diffeomorphisms} (chapter 4). Translation of the 2007 Spanish edition. Chicago Lectures in Mathematics. University of Chicago Press, Chicago, IL, 2011. 

\bibitem[Na2]{Na2}
A.~\textsc{Navas} ---
\textit{Sur les rapprochements par conjugaison en dimension $1$ et classe $\CC^1$},  arXiv:1208.4815.

\bibitem[RoRo]{RoRo}
H.~\textsc{Rosenberg}; R.~\textsc{Roussarie} ---
\textit{Some remarks on stability of foliations.}
J. Differential Geometry 10 (1975), 207--219.

 
\bibitem[Se]{Se}
F.~\textsc{Sergeraert} ---
\textit{Feuilletages et difféomorphismes infiniment tangents à~l'iden\-tité}. \   
Invent. Math. \textbf{39} (1977), 253--275.

\bibitem[Sz]{Sz} 
G.~\textsc{Szekeres} ---
\textit{Regular iteration of real and complex functions}.  \ 
Acta Math. \textbf{100} (1958), 203--258.

\bibitem[Ta]{Ta}
F.~\textsc{Takens} ---
\textit{Normal forms for certain singularities of vector fields}. \ 
Ann. Inst. Fourier \textbf{23} (1973), 163--195.

\bibitem[Ts]{Ts} 
T.~\textsc{Tsuboi} ---
\textit{Hyperbolic compact leaves are not $C^1$-stable. } \
 Geometric study of foliations (Tokyo, 1993), 437--455, World Sci. Publ., River Edge, NJ, 1994. 

\bibitem[Wo]{Wo}
J.~\textsc{Wood} ---
\textit{Foliations on 3-manifolds}. \
 Ann. of Math. \textbf{89}  (1969), 336--358.


\bibitem[Yo]{Yo}
J-C.~\textsc{Yoccoz} ---
\textit{Centralisateurs et conjugaison différentiable des difféomorphismes du cercle}. \
Astérisque 231.


\end{thebibliography}
\end{document}